\documentclass[12pt]{amsart}

\usepackage[english]{babel}

\usepackage[pdftex,textwidth=400pt,marginratio=1:1]{geometry}

\usepackage{amsfonts}

\usepackage[dvips]{graphics}

\usepackage{amsmath}

\usepackage{amsthm}

\usepackage{amssymb}

\usepackage{eufrak}

\usepackage{cancel}

\usepackage{color}

\usepackage[curve]{xypic}

\newtheorem{theorem}{Theorem}[section]

\newtheorem{proposition}[theorem]{Proposition}

\newtheorem{lemma}[theorem]{Lemma}

\newtheorem{remark}[theorem]{Remark}

\theoremstyle{definition}

\theoremstyle{property}

\makeatletter
\newcommand{\contraction}[5][1ex]{%
  \mathchoice
    {\contraction@\displaystyle{#2}{#3}{#4}{#5}{#1}}%
    {\contraction@\textstyle{#2}{#3}{#4}{#5}{#1}}%
    {\contraction@\scriptstyle{#2}{#3}{#4}{#5}{#1}}%
    {\contraction@\scriptscriptstyle{#2}{#3}{#4}{#5}{#1}}}%
\newcommand{\contraction@}[6]{%
  \setbox0=\hbox{$#1#2$}%
  \setbox2=\hbox{$#1#3$}%
  \setbox4=\hbox{$#1#4$}%
  \setbox6=\hbox{$#1#5$}%
  \dimen0=\wd2%
  \advance\dimen0 by \wd6%
  \divide\dimen0 by 2%
  \advance\dimen0 by \wd4%
  \vbox{%
    \hbox to 0pt{%
      \kern \wd0%
      \kern 0.5\wd2%
      \contraction@@{\dimen0}{#6}%
      \hss}%
    \vskip 0.5ex
    \vskip\ht2}}

\newcommand{\contraction@@}[3][0.05em]{%
  \hbox{%
    \vrule width #1 height 0pt depth #3%
    \vrule width #2 height 0pt depth #1%
    \vrule width #1 height 0pt depth #3%
    \relax}}
\makeatother

\DeclareFontFamily{OT1}{rsfs}{}
\DeclareFontShape{OT1}{rsfs}{n}{it}{<-> rsfs10}{}
\DeclareMathAlphabet{\curly}{OT1}{rsfs}{n}{it}
\newcommand\A{\mathcal A}
\newcommand\I{\curly I}

\newcommand\LL{\mathbb L}
\renewcommand\S{\mathcal S}
\renewcommand\O{\mathcal O}
\newcommand\PP{\mathbb P}
\newcommand\Db{\overline{\!D}}
\newcommand\Mb{\,\overline{\!M}}
\newcommand\MMb{\,\overline{\!\mathcal M}}
\newcommand\Xb{\,\overline{\!X}}
\newcommand\XXb{\,\overline{\!\mathcal X}}
\newcommand\C{\mathbb C}
\newcommand\II{\mathbb I}
\newcommand\Q{\mathbb Q}
\newcommand\R{\mathbb R}
\newcommand\Z{\mathbb Z}
\newcommand\m{\mathfrak m}

\newcommand{\Rt}[1]{\stackrel{#1\,}{\longrightarrow}}
\newcommand\To{\longrightarrow}
\newcommand\into{\hookrightarrow}
\newcommand\Into{\ar@{^(->}[r]<-.3ex>}

\renewcommand\_{^{}_}

\newcommand\bull{{\scriptscriptstyle\bullet}}
\newcommand\udot{^\bull}

\newcommand\rk{\operatorname{rank}}
\newcommand\tr{\operatorname{tr}}

\newcommand\ev{\operatorname{ev}}
\renewcommand\div{\operatorname{div}}
\newcommand\id{\operatorname{id}}

\newcommand\Hom{\operatorname{Hom}}
\renewcommand\hom{\curly H\!om}
\newcommand\Ext{\operatorname{Ext}}
\newcommand\ext{\curly Ext}

\newcommand\ob{\operatorname{ob}}
\newcommand\Pic{\operatorname{Pic}}

\newcommand\Spec{\operatorname{Spec}\,}
\newcommand\Hilb{\operatorname{Hilb}}

\newcommand\beq[1]{\begin{equation}\label{#1}}
\newcommand\eeq{\end{equation}}
\newcommand\beqa{\begin{eqnarray*}}
\newcommand\eeqa{\end{eqnarray*}}

\DeclareRobustCommand{\SkipTocEntry}[4]{}

\begin{document}

\title[Curve counting on surfaces I: theory]{Reduced classes and curve counting on surfaces I: theory \vspace{-1mm}}

\author[M. Kool and R. P. Thomas]{Martijn Kool and Richard Thomas  \\ with an Appendix written with Dmitri Panov \vspace{-2mm}}

\begin{abstract}
We develop a theory of \emph{reduced} Gromov-Witten and stable pair invariants of surfaces and their canonical bundles.

We show that classical Severi degrees are special cases of these invariants.
This proves a special case of the MNOP conjecture, and allows us to generalise the G\"ottsche conjecture to the non-ample case. In a sequel we prove this generalisation.

We prove a remarkable property of the moduli space of stable pairs on a surface. It is the zero locus of a section of a bundle on a smooth compact ambient space, making calculation with the reduced virtual cycle possible.
\vspace{-6mm}
\end{abstract}

\maketitle
\thispagestyle{empty}
\renewcommand\contentsname{\vspace{-8mm}}
\tableofcontents

\section{Introduction}

\subsection*{Motivation}
Fix a nonsingular projective surface $S$ and a homology class $\beta\in H_{2}(S,\Z)$. There are various ways of counting holomorphic curves in $S$ in class $\beta$; in this paper we focus on Gromov-Witten invariants \cite{Beh, LT} and stable pairs \cite{PT1, Ott}. Since these are deformation invariant they must vanish in class $\beta$ if there exists a deformation of $S$ for which the Hodge type of $\beta$ is not $(1,1)$. We can see the origin of this vanishing without deforming $S$ as follows.

For simplicity work in the simplest case of an embedded curve $C\subset S$ with normal bundle $N_C=\O_C(C)$. As a Cartier divisor, $C$ is the zero locus of a section $s_C$ of a line bundle $L:=\O_S(C)$, giving the exact sequence
$$
0\To\O_S\Rt{s_C} L\To N_C\To0.
$$
The resulting long exact sequence describes the relationship between first order deformations and obstructions $H^0(N_C),\,H^1(N_C)$ of $C\subset S$, and the deformations and obstructions $H^1(\O_S),\,H^2(\O_S)$ of the line bundle $L\to S$:
\begin{align} \nonumber
0\To H^0(L)\big/\langle s_C\rangle &\To H^0(N_C)\To H^1(\O_S)\To H^1(L) \\
& \To H^1(N_C)\To H^2(\O_S)\To H^2(L)\To0. \label{semireg}
\end{align}
The resulting ``semi-regularity map" \cite{KS} $H^1(N_C)\to H^2(\O_S)=H^{0,2}(S)$ takes obstructions to deforming $C$ to the ``cohomological part" of these obstructions. Roughly speaking, if we deform $S$, we get an associated obstruction in $H^1(N_C)$ to deforming $C$ with it; its image in $H^{0,2}(S)$ is the $(0,2)$-part of the cohomology class $\beta\in H^2(S)$ in the deformed complex structure. Thus it gives the obvious cohomological obstruction to deforming $C$: that $\beta$ must remain of type $(1,1)$ in the deformed complex structure on $S$.


In particular, when $S$ is fixed, obstructions lie in the kernel of $H^1(N_C)\to H^2(\O_S)$. More generally, if we only consider deformations of $S$ for which $\beta$ remains $(1,1)$ then the same is true. And when $h^{0,2}(S)>0$ but $H^2(L)=0$, the existence of this trivial $H^{0,2}(S)$ piece of the obstruction sheaf guarantees that the virtual class vanishes.

So it would be nice to restrict attention to surfaces and classes $(S,\beta)$ inside the Noether-Lefschetz locus,\footnote{The locus of surfaces $S$ for which $\beta\in H^2(S)$ has type $(1,1)$; for more details see \cite{Voi, MP}.} defining a new obstruction theory using only the kernel of the semi-regularity map.\footnote{For embedded curves this means we use the obstruction space $H^1(L)$ to deforming sections of $L$. We have been able to remove the obstructions $H^2(\O_S)$ to deforming $L$ since the space of line bundles is smooth over the Noether-Lefschetz locus.} Checking that this kernel really defines an obstruction theory in the generality needed to define a virtual cycle -- i.e. for deformations to all orders, over an arbitrary base, of possibly non-embedded curves -- has proved difficult; there is a hotchpotch of results in different cases \cite{BF2, BL1, Blo, BuF, Don, IM, KL, Lee, Li, Liu, Man, MP, MPT, OP, Ran, Ros, Sch, STV}. Here we give quite a general construction using a mixture of some of these methods.\footnote{Since this paper appeared Jon Pridham has found a more general solution \cite{Pri} using derived deformation theory, $\infty$-stacks, etc. His result is broader and more natural, but our methods are much more elementary.}

\subsection*{Our results}\ \smallskip

\noindent\textbf{Surfaces.}
For stable pairs on $S$ we get optimal results. We show that the kernel of the semi-regularity map gives a reduced perfect obstruction theory, virtual cycle and invariants whenever
\beq{condition}
H^2(L)=0\text{ for effective line bundles $L$ with }c_1(L)=\beta.
\eeq
Equivalently, by Serre duality, the condition is that there is no curve in class $\beta$ which is contained in a canonical divisor of $S$.
This condition is necessary to ensure the semi-regularity map \eqref{semireg} is surjective.

For Gromov-Witten theory,
multiple covers complicate the situation, but we are able to prove the same result for the moduli space of stable maps when
\beq{condition2}
H^1(T_S)\Rt{\cup\beta}H^2(\O_S) \text{ is surjective}.
\eeq
Here $\beta\in H^1(\Omega_S)$ and we use the pairing $\Omega_S\otimes T_S\to\O_S$.
Condition \eqref{condition2} implies \eqref{condition}: for any $L=\O(C)$ in class $\beta$, the map $\cup\beta$ factors through $$H^1(T_S)\To H^1(\O_C(C))\To H^2(\O_S),$$ so surjectivity implies that $H^2(L)=0$ by the exact sequence \eqref{semireg}.

Condition \eqref{condition2} is a transversality assumption on the moduli space of surfaces $S$. It asks that the $h^{2,0}(S)$ equations cutting out the Noether-Lefschetz locus 
$$
\int_\beta\sigma_i=0, \qquad \{\sigma_i\colon i=1,\ldots,h^{2,0}(S)\}\mathrm{\ a\ basis\ for\ }H^{2,0}(S),
$$
are transverse to $0$. In particular if the moduli space of surfaces is smooth, it asks that the Noether-Lefschetz locus be smooth of the expected codimension $h^{2,0}(S)$. For example, we note that for degree $d \geq 4$ surfaces $S\subset\PP^3$, it is \emph{almost always} satisfied in a precise sense \cite[Section 3]{Ki}. \medskip

\noindent\textbf{Method.}
We embed $S$ as the central fibre of an \emph{algebraic twistor family}\footnote{This is an outrageous abuse of notation, motivated by the $S=K3$ case where it is a first order neighbourhood of the central fibre in the twistor family used in \cite{BL1}. It should also be noted that the family is not canonical, but involves choices.} $$\S_B\to B.$$ Here $B$ is a first order Artinian neighbourhood of the origin in a certain $h^{0,2}(S)$-dimensional family of first order deformations of $S$. 

In Section \ref{red} we show that Condition \eqref{condition2} implies that the relative moduli space of curves (stable maps or stable pairs) on the fibres of $\S_B$ is in fact the moduli space of curves on the central fibre. We then show that the natural perfect obstruction theory of the family (its relative perfect obstruction theory made absolute) is isomorphic to the kernel of the semi-regularity map on the standard obstruction theory. Thus the latter, which is canonical, can indeed be used as an obstruction theory, giving a definition of reduced curve counting invariants. These coincide with the usual invariants when $h^{0,2}(S)=0$.

This use of relative moduli spaces means that we require Condition \eqref{condition2} only for $\beta$. When \eqref{condition2} holds also for all $\beta'<\beta$
then the absolute moduli space of curves in $\S_B$ also coincides with the moduli space for $S$, and the ordinary Gromov-Witten invariants of the total space $\S_B$ can be expressed in terms of the reduced Gromov-Witten invariants of $S$ with so-called $K_S$-twisted $\lambda$-class insertions. (These are Chern classes of the virtual bundle whose fibre over $f\colon C\to S\subset\S_B$ is $R\Gamma(f^*K_S)$. For the K3 case see \cite{MPT}.) \medskip

\noindent\textbf{Threefolds.}
We also work on the Calabi-Yau 3-fold $$(X=K_S)\ \curvearrowleft\C^*,$$
always using Condition \eqref{condition2} in this case. Again we get a reduced obstruction theory, and use $\C^*$-localisation to define reduced residue Gromov-Witten and stable pair invariants. The former come entirely from the moduli space of stable maps to $S$ itself and include the reduced invariants of $S$. More generally they include some $\lambda$-classes twisted by $K_S$. The moduli space $P_X^{\C^*}$ of stable pairs fixed by the $\C^*$-action, however, contains stable pairs not scheme-theoretically supported on $S$, so the moduli space is bigger than the moduli space $P_S$ of stable pairs on $S$. But $P_S\subset P_X^{\C^*}$ forms a connected component, so the reduced residue invariants of $X$ contain a contribution coming entirely from $S$. This contribution includes the reduced invariants of $S$, and more generally signed virtual Euler characteristics of loci in $P_S$ satisfying incidence conditions. In the sequel \cite{KT2}, we describe conditions under which $P_S$ is all of $P_X^{\C^*}$. \medskip

\noindent\textbf{Insertions and Severi degrees.}
In Section \ref{insert} we develop a careful treatment of insertions in these theories. We then prove various folklore results about them, such as the link to the Picard variety and sublinear systems. (The existing literature only handles these issues within symplectic geometry.)

In Section \ref{nodal} we use this to show that with the right insertions, the reduced Gromov-Witten invariants recover the Severi degrees
\beq{Sevdeg}
n_\delta(L):=\deg\ \overline{\{C\in|L|\colon C \mathrm{\ has\ }\delta
\mathrm{\ nodes}\}}\ \subset\ |L|
\eeq
of appropriately ample linear systems $|L|$ on $S$.

In this case we are also able to show that only $P_S\subset P_X^{\C^*}$ contributes to the stable pair invariants of $X$. Therefore the 3-fold MNOP conjecture \cite{MNOP} applied to $X$ (and extended to the reduced $\C^*$-localised invariants) predicts that the $n_\delta(L)$ should also be expressible as a precise combination of reduced stable pair invariants of $S$. We prove this in Theorem \ref{GottchaP}.

This also points to a definition of \emph{virtual Severi degrees} in the non-ample case via reduced Gromov-Witten or stable pair invariants. Moreover, we give an extension of the G\"ottsche conjecture to these virtual numbers: i.e. that they should be computed by the G\"ottsche polynomials even when the latter are not obviously enumerative. In the sequel \cite{KT2} we compute the resulting stable pair invariants and prove this version of the conjecture.
\medskip

\noindent\textbf{Stable pairs as a zero locus.}
In Appendix \ref{Dmitri}, written with Dmitri Panov, we give an alternative, more direct construction of the reduced stable pair theory on a surface $S$, without reference to $\S_B$. This is achieved by realising the moduli space in the following way. We first take the zero locus of a natural section of a bundle over a smooth ambient space, then we take the zero locus of a section of another bundle over that.\footnote{After posting this paper we discovered that the first step here was used many years ago by D\"urr, Kabanov and Okonek \cite{DKO}, giving a description of the virtual cycle on the Hilbert scheme of curves in $S$. We then extend this to stable pairs over those curves.}

This is an unusual phenomenon. Having a \emph{local} description of a moduli space as the zero locus of a section of a bundle $E$ over a smooth ambient space $A$ with $$
\dim A-\rk E=v
$$
is basically equivalent to having a perfect obstruction theory of virtual dimension $v$.
But having such a description \emph{globally}, for some \emph{compact} $A$ and the same $v$, is extremely rare. It is also extremely desirable: it means the pushforward to $A$ of the virtual cycle is Poincar\'e dual to $c_{top}(E)$, with which one can try to calculate.

We exploit this in the sequel \cite{KT2} to calculate the reduced stable pair invariants in terms of universal formulae in topological numbers\footnote{In contrast the non-reduced stable pair invariants of $S$ do not have such a simple form in general, depending on the Seiberg-Witten invariants of $S$. See \cite{Ko}, where a duality formula is also obtained, and the MNOP conjecture is related to Taubes' SW=GW correspondence.} of $(S,\beta)$. In \cite{PTKKV} this also provides one of the foundations of a computation of the full stable pairs theory of the twistor family of a K3 surface. Via Pandharipande and Pixton's recent proof of the MNOP conjecture for many 3-folds, this then gives a proof of the famous KKV formula for the Gromov-Witten invariants of K3 surfaces in all genera, degrees and for all multiple covers.

Except for genus-0 Gromov-Witten calculations in complete intersections in convex varieties, we know of no other moduli problem where such direct calculation is possible. Usually obtaining explicit results is very complicated, involving various difficult degeneration and localisation tricks.
\medskip

\noindent\textbf{Organisation.}
The paper is organised as follows. Most results are proved twice, once for stable maps, and an analogous result for stable pairs.
We work out the reduced obstruction theories in Section \ref{red} assuming Condition \eqref{condition2}.
In Appendix \ref{Dmitri} we use an easier construction to show that Condition \eqref{condition2} can be replaced by \eqref{condition} for stable pairs on $S$ only. In Section \ref{invts} we define the corresponding reduced invariants for $S$, the reduced residue invariants of $X=K_S$, and we show the latter contain the information of the former. Section \ref{insert} deals with insertions and linear systems of curves in $S$. Section \ref{nodal} discusses the application to Severi degrees and the MNOP conjecture for reduced invariants with many point insertions.

In summary, we describe invariants incorporating the following.
\begin{itemize}
\item Reduced Gromov-Witten and stable pair invariants of $S$,
\item Gromov-Witten and stable pair invariants of $S$ when $H^{2,0}(S)=0$,
\item Reduced equivariant Gromov-Witten and stable pair invariants of $K_S$,
\item Equivariant invariants of $K_S$ when $H^{2,0}(S)=0$,
\item Reduced GW invariants of $S$ with $K_S$-twisted $\lambda$-class insertions,
\item Relative Gromov-Witten and stable pair invariants of $\S_B\to B$,
\item Absolute Gromov-Witten invariants of $\S_B$,
\item Absolute equivariant Gromov-Witten invariants of $K_{\S_B/B}$,
\item Severi degrees and virtual Severi degrees.
\end{itemize} \smallskip

\noindent \textbf{Acknowledgements}. This project originated with a suggestion of Daniel Huybrechts, for which we are very grateful. We would like to thank Nicolas Addington, Jim Bryan, Davesh Maulik, Vivek Shende, an excellent referee and particularly Rahul Pandharipande for their assistance. We owe an obvious \nopagebreak
intellectual debt to the papers \cite{BL1, MPT}. Both authors were supported by an EPSRC programme grant EP/G06170X/1.
\setlength{\textheight}{21cm}
\newpage

\subsection*{Notation}
Throughout we keep largely to the following notation. \vspace{2mm}

\begin{tabular}{ll}
$S$ & a smooth projective surface \\
$X,\ \Xb$ & the total space of the canonical bundle $K_S$ of $S$ and its \\ & projective completion $\PP(K_S\oplus\O_S)$ respectively \\
$V\supset B$ & choice \eqref{choice} of $h^{0,2}(S)$-dimensional subspace of $H^1(T_S)$ \\ & and first order thickening \eqref{SB} of its origin \\
$\S_B,\ \mathcal X_B,\ \XXb_B$ &  \eqref{SB} the algebraic twistor family of $S$ over the Artinian \\ & base $B$, its relative canonical bundle and completion \\
$\iota,\,q$ & inclusion map of $S$ into $X$, $\Xb$ and projection $q\colon X,\Xb\to S$ \\
$\beta$ & a class in $H^2(S,\Z)$, usually of Hodge type $(1,1)$ \\
$h$ & the arithmetic genus of curves in class $\beta$, determined by \\ & the adjunction formula $2h-2=\beta^2-\int_\beta c_1(S)$ \\
$L$ & a line bundle on $S$ with $c_1(L)=\beta$ \\
Condition \eqref{condition} & $H^2(L)=0$ for effective line bundles $L$ with $c_1(L)=\beta$ \\
Condition \eqref{condition2} & $H^1(T_S)\Rt{\cup\beta}H^2(\O_S)$ is surjective\\
$\gamma\_i$ & \eqref{gamm} integral basis $\gamma\_1,\ldots,\gamma\_{b_1(S)}$ of $H_1(S,\Z)/$torsion, \\ & oriented with respect to the complex structure on $H_1(S,\R)$ \\
div & \eqref{maps} divisor class in $S$ of a stable map \\
det & \eqref{maps} line bundle associated to above divisor class \\
$[\gamma]$ & Poincar\'e dual of homology class $\gamma$ \\
$\mathcal P$ & Poincar\'e line bundle on $S\times\Pic(S)$ \\
$\mathfrak t,\ t=c_1(\mathfrak t)$ & one dimensional irreducible representation of $T=\C^*$ of \\ & weight 1, and generator of $H^*(B\C^*,\Z)=\Z[t]$ respectively \\
$\Mb_{\!g,n}(S,\beta)$ & moduli space of stable maps from connected genus $g$ \\ & curves with $n$ marked points to $S$ in class $\beta$ \\
$\Mb_{\!g}(S,\beta)$ & as above but without marked points, i.e. $n=0$ \\
$\Mb_{\!g,n}(S,\PP^\delta)$ & \eqref{MbLdef} stable maps whose divisor class lies in a given linear \\ & system $\PP^\delta\subset|L|$ \\
$P_n(S,\beta)$ & moduli space of stable pairs $(F,s)$ on $S$ with curve class \\ & $\beta$ and holomorphic Euler characteristic $\chi(F)=n$ \\
$R_{g,\beta}$ & \eqref{Rgbeta} reduced Gromov-Witten invariant $\in\Q$ \\
$\curly R_{g,\beta}$ & \eqref{RgbetaX} reduced residue Gromov-Witten invariant $\in\Q(t)$ \\
$P_{n,\beta}^{red}$ & \eqref{SPS} reduced stable pair invariant $\in\Z$ \\
$\curly P_{\!n,\beta}^{red}$ & reduced residue stable pair invariant $\in\Z(t)$; see \eqref{pairsdefn} for \\ & $X$ and \eqref{Pres} for $S$ \\
$r_{g,\beta}$ & \eqref{redGV} reduced residue Gopakuma-Vafa BPS invariants \\
$g=g(C)$ & arithmetic genus of a curve $C$ \\
$\overline g:=g(\overline C)\quad\ \ $ & geometric genus = arithmetic genus of normalisation $\overline C$ \\
$n_\delta(L)$ & \eqref{Gtt} Severi degree counting $\delta$-nodal curves in \\ & $\delta$-dimensional linear subsystems $\PP^\delta$ of $|L|$ \\
$H_\beta\subset H_\gamma$ & \eqref{+A} $\Hilb_\beta(S)$ embedded in $\Hilb_{\gamma:=[A]+\beta}(S)$ by $C\mapsto C+A$
\end{tabular}
\setlength{\textheight}{19.5cm}

\section{Reduced obstruction theories}\label{red}

\subsection{The algebraic twistor family $\S_B$}

Fix a projective surface $S$, and let $\m$ denote the maximal ideal at the origin $0\in H^1(T_S)$. The first order neighbourhood of the origin
$$
\Spec\O_{H^1(T_S)}/\m^2
$$
has a cotangent sheaf whose restriction to the origin is $H^1(T_S)^*$. Over this Artinian space lies a tautological flat family of surfaces $\S$ with Kodaira-Spencer class the identity in $H^1(T_S)^*\otimes H^1(T_S)$ parametrising the extension
$$
0\To H^1(T_S)^*\otimes\O_S\To\Omega_\S|_S\To\Omega_S\To0.
$$

Now fix a class $\beta\in H^{1,1}(S)\cap(H^2(S,\Z)/$torsion) for which $H^1(T_S)\Rt{\cup\beta}H^2(\O_S)$ is a surjection. Picking a splitting we get a $h^{0,2}(S)$-dimensional subspace $V\subset H^1(T_S)$ such that
\beq{choice}
\cup\beta\colon V\to H^2(\O_S)
\eeq
is an isomorphism.
Restricting the family $\S$ to $V$ gives a flat family
\beq{SB}
\S_B\ \text{ over }\ B:=\Spec\O_V/\m^2=\Spec\!\big(\C\oplus V^*\big)
\eeq
and an exact sequence of K\"ahler differentials
\beq{cotangent}
0\To\Omega_B|_S\To\Omega_{\S_B}|_S\To\Omega_S\To0.
\eeq
In the first term we have suppressed the pullback map from the base; the result is $V^*\otimes\O_S$. The extension class of \eqref{cotangent} -- the Kodaira-Spencer class in $H^1(T_S\otimes V^*)=\Hom(V,H^1(T_S))$ of the family $\S_B$ -- induces an isomorphism
\beq{nondegen}
\xymatrix{
V\ar[r]\ar@/_{3ex}/[rr]_(.45)\simeq & H^1(T_S)\ar[r]^{\cup\beta\ } & H^2(\O_S).
}
\eeq
There is a canonical isomorphism \cite[Proposition 3.8]{Blo}
$$
H^2_{dR}(\S_B/B)\cong H^2(S,\C)\otimes_\C\O_B.
$$
So corresponding to $\beta\otimes1$ we get the \emph{horizontal lift of} $\beta$:
$$
\beta_B\in H^2_{dR}(\S_B/B).
$$
By projection we obtain a class $[\beta_B]^{0,2}\in H^2_{dR}(\S_B/B)/F^1H^2_{dR}(\S_B/B)$, where $F^1H^2_{dR}(\S_B/B)$ is the part of the Hodge filtration defined by $\Omega^{\ge1}_{\S_B/B}$. The scheme theoretic Noether-Lefschetz locus in $B$ is defined to be the zero locus of $[\beta_B]^{0,2}$. The family $\S_B$ was constructed precisely to ensure the following.

\begin{lemma} \label{NLlem}
The Noether-Lefschetz locus in $B$ is just the closed point $0\in B$.
\end{lemma}

\begin{proof}
Since $\beta$ has type $(1,1)$ on the central fibre $S$, $[\beta_B]^{0,2}$ certainly vanishes at the origin $0$. Next pick any nonzero tangent vector $v\in V$, thus defining a subscheme $B_v\subset B$ by intersecting $B$ with the span $\C v\subset V$ (equivalently, define $B_v$ via the ideal $\langle v\rangle^\perp\subset V^*\subset\O_B$). By \cite[Proposition 4.2]{Blo}, $B_v$ lies in the Noether-Lefschetz locus if and only if
\beq{zero?}
\nabla(\beta)=0\ \text{ in } H^2(\O_S).
\eeq
Here $\nabla$ \cite[(4.1)]{Blo} is the map given by cup product with the Kodaira-Spencer class of $\S_{B_v}$, i.e. the image of $v\in V$ under the first arrow of \eqref{nondegen}. By \eqref{nondegen} then, \eqref{zero?} does not hold, so $B_v$ is not in the Noether-Lefschetz locus.
\end{proof}

Let $j\colon S\into\S_B$ denote the inclusion of the central fibre, and denote by $\Mb_{\!g,n}(\S_B/B,\beta_B)\to B$ the moduli space of stable maps of connected genus $g$ curves with $n$ marked points to the fibres of $\S_B\to B$. 
\begin{proposition} \label{stablemapiso}
Recall that we are assuming Condition \eqref{condition2}. Then the natural morphism of stacks $j_*\colon\Mb_{\!g,n}(S, \beta) \to\Mb_{\!g,n}(\S_B/B, \beta_B)$ is an isomorphism.
\end{proposition}

\begin{proof}
Since $\Mb_{\!g,n}(\S_B/B,\beta_B)\times_B\{0\}\cong
\Mb_{\!g,n}(S,\beta)$ we need only prove the following.

Suppose we have an Artinian scheme $A$ with a morphism to $B$, a proper flat family $C \to A$ and a $B$-morphism $h\colon C \to \S_B$ which pulls back to a stable map $h_{0} : C_{0} \to S$ satisfying $h_{0*}[C_{0}] = \beta$. Then we want to show that $A\to B$ factors through $0 \in B$.

Now define $\S_A/A:=\S_B\times_BA$ with horizontal class
$$
\beta_A:=\beta\otimes1\in H^2_{dR}(\S_A/A)\cong H^2(S,\C)\otimes_\C\O_A.
$$
It is the pullback of $\beta_B$ via $A\to B$. \cite[Proposition 5.6]{Blo} defines the class $h_*[C] = \beta_A\in F^1H^{2}_{dR}(\S_A/A)$. In particular $[\beta_A]^{0,2}\in H^{2}_{dR}(\S_A/A)/
F^1H^{2}_{dR}(\S_A/A)$
is zero, so that the image of $A$ lies in the zero locus of $[\beta_B]^{0,2}$, which by Lemma \ref{NLlem} is scheme theoretically just the closed point $0\in B$.
\end{proof}

We are also interested in the threefold $X=K_S$ that is the total space of the canonical bundle of $S$. For technical reasons it is often convenient to work on its projective completion $\Xb:= \mathbb{P}(K_S \oplus \O_S)$.

Let $\beta \in H^{2}(S)$ and $\iota\colon S \hookrightarrow \Xb$ be the inclusion of the zero-section. Since $\Xb$ is a $\mathbb{P}^{1}$-bundle over $S$, there are canonical isomorphisms $H^{1}(T_{\Xb}) \cong H^{1}(T_S)$ and $H^{i,j}(S)\cong H^{i+1,j+1}(\Xb)$ when $i+j=2$, intertwining $\cup \beta$ with $\cup\,\iota_{*} \beta$. Associated to the family $\S_B\to B$ \eqref{SB} we also get families of 3-folds $\mathcal X_B\to B$ and $\XXb_B\to B$, with natural inclusions from $\S_B$ which we also denote by $\iota$. As before we let $j$ denote any of the inclusions of the central fibres $S,X,\Xb$ into the families $\S_B,\mathcal X_B,\XXb_B$.

In addition to the moduli space of stable maps $\Mb_{\!g,n}(\Xb,\iota_*\beta)$ to $\Xb$, we will also want to use moduli spaces of stable pairs \cite{PT1} on $S$ and $\Xb$.

Denote by $P_{n}(\Xb, \iota_{*} \beta)$ the (fine) moduli space of stable pairs $(F,s)$ on $\Xb$ with $ch_2(F) = \iota_{*} \beta$ and holomorphic Euler characteristic $\chi(F) = n$. Similarly
$P_{n}(\XXb_B / B, \iota_{*} \beta_B)$ is
the relative moduli space of stable pairs on the fibres of $\XXb_B\to B$. We can repeat these definitions for $S$ and $\S_B/B$ for stable pairs $(F,s)$ with $c_1(F) = \beta$ and $\chi(F) = n$.

\begin{proposition} \label{stablepairiso}
Assuming Condition \eqref{condition2}, the natural morphisms
\begin{enumerate}
\item $j_*\colon\Mb_{\!g,n}(S, \beta) \to\Mb_{\!g,n}(\S_B/B, \beta_B)$,
\item $j_*\colon\Mb_{\!g,n}(\Xb,\iota_*\beta) \to\Mb_{\!g,n}(\XXb_B/B,
\iota_*\beta_B)$,
\item $j_*\colon P_n(S, \beta) \to P_n(\S_B/B,\beta_B)$, and
\item $j_*\colon P_{n}(\Xb, \iota_{*} \beta) \to P_{n}(\XXb_B / B, \iota_{*} \beta_B)$
\end{enumerate}
are isomorphisms.
\end{proposition}

\begin{proof}
The first is Proposition \ref{stablemapiso}. The same proof, with $h_*[C]$ replaced by $c_1(F)$, gives (3).

For (2) and (4), the proof is the same once we replace $h_*[C]$ by $ch_2(F)$ and $H^2_{dR}(\S_B/B)/
F^1H^2_{dR}(\S_B/B)$ by $H^4_{dR}(\XXb_B/B)/
F^2H^4_{dR}(\XXb_B/B)$.
\end{proof}

\subsection{Stable maps} \label{GWsemireg}

In this section we abbreviate the notation $\Mb_{\!g,n}(S,\beta)$ for the moduli space of stable maps with $n$ marked points to $\Mb$. 
When $n=0$ and $f\colon C\into S$ is an embedded curve with normal bundle $N_C$, the natural deformation-obstruction theory for stable maps is
$$
E\udot:=R\Gamma(N_C)^\vee
$$
at the point $f\in\Mb$. This naturally extends to a 2-term complex over $\Mb$ with a morphism to the truncated cotangent complex
$\LL_{\Mb}:=\tau^{\ge-1}L_{\Mb}\udot$ of $\Mb$. The semi-regularity map $h^1((E\udot)^\vee)=H^1(N_C)\to H^2(\O_S)$ of \eqref{semireg} is Serre dual to the composition
\beq{semireg2}
H^0(K_S)\To H^0(f^*\Omega_S\otimes f^*\Omega_S)\To H^0(N_C^*\otimes\omega_C),
\eeq
where all of the maps are the obvious ones.

For a general stable map $f\colon(C,x_1,\ldots,x_n)\to S$ we replace $N_C^*$ by the complex\footnote{$f^*\Omega_S$ is placed in degree $0$. When $C$ is embedded and $n=0$ the arrow is surjective so the complex is indeed quasi-isomorphic to its kernel $N_C^*$.} $\{f^*\Omega_S\to
\Omega_C(x_1+\ldots+x_n)\}$. More globally, denote the universal curve by
\beq{diag}
\xymatrix@R=15pt{
\mathcal C \ar[r]^{f} \ar[d]_{\pi}  & S \\
\Mb.}
\eeq
Then the perfect obstruction theory of Gromov-Witten theory is \cite{Beh, BF1}
\beq{Edot}
E\udot:= \left( R\pi_{*} R\hom\big(\big\{f^{*}\Omega_S \to \Omega_{\mathcal C/\Mb}^{\mathrm{log}}\big\},\O_{\mathcal C}\big) \right)^{\vee}\To\LL_{\Mb},
\eeq
where 
\beq{logOm}
\Omega_{\mathcal C/\Mb}^{\log}:= \Omega_{\mathcal C/\Mb}(x_{1} + \cdots + x_{n})
\eeq
and $x_{1}, \ldots, x_{n}$ are the sections of the universal curve defining the marked points. Since marked points are always \emph{smooth} points of the curve, \eqref{logOm} is the sheaf of logarithmic one forms.  Letting $ \omega_{\mathcal C/\Mb}$ be the relative dualising line bundle, consider the composition of the following standard maps
\beq{cxs}
\xymatrix@C=4pt@R=18pt{
f^*K_S \ar[rr] && f^*\Omega_S\otimes f^*\Omega_S \ar[d] \ar@{-->}[dr] \\
&& \big\{f^*\Omega_S \ar[r] & \Omega_{\mathcal C/\Mb}\big\}\otimes\Omega_{\mathcal C/\Mb} \ar[rr] && \big\{f^*\Omega_S \to \Omega_{\mathcal C/\Mb}^{\log}\big\}
\otimes\omega_{\mathcal C/\Mb}.}
\eeq
The composition $f^*K_S\to\Omega_{\mathcal C/\Mb}\otimes\Omega_{\mathcal C/\Mb}$ of the first arrow and the dotted arrow is zero since it factors through $\Lambda^2\Omega_{\mathcal C/\Mb}=0$. Therefore \eqref{cxs} defines a map of complexes
$$
f^*K_S\To\big\{f^*\Omega_S\to\Omega_{\mathcal C/\Mb}^{\log}\big\}\otimes\omega_{\mathcal C/\Mb}.
$$
Composing with $H^0(K_S)\to H^0(f^*K_S)$ and applying $R\pi_*$ we get the (dual) semi-regularity map
\beq{reg}
H^0(K_S)\otimes\O_{\Mb}\To E\udot[-1],
\eeq
which reduces to \eqref{semireg2} when $f$ is an embedding and there are no marked points. Dualising and taking $h^1$ gives
\beq{surjec}
\ob\To H^2(\O_S)\otimes\O_{\Mb},
\eeq
where $\ob:=h^1((E\udot)^\vee)$ is the obstruction sheaf on $\Mb$. We will see that this is a surjection in the proof of Theorem \ref{reducedGW}.

We define the \emph{reduced obstruction theory} of $\Mb$ to be the cone on the map \eqref{reg}:
$$
E_{red}\udot:=\mathrm{Cone}\big(H^0(K_S)\otimes\O_{\Mb}[1]\to E\udot\big).
$$
Its name is justified by the next result.

\begin{theorem} \label{reducedGW}
Suppose that $\beta\in H^2(S,\Z)$ is a $(1,1)$ class such that $H^1(T_S)
\Rt{\cup\beta}H^2(\O_S)$ is surjective. Then there is a perfect obstruction theory\footnote{While $E\udot_{red}$ is canonical, we are currently unable to prove that the map $E_{red}\udot\to\LL_{\Mb}$ is independent of the choice \eqref{choice} made to define the algebraic twistor family $\S_B$. But the induced virtual cycle is 
canonical, since it only depends on the K-theory class of $E\udot_{red}$ \cite{Pid, Sie}.}
$E_{red}\udot\to\LL_{\Mb}$ for $\Mb_{\!g,n}(S,\beta)$ with virtual dimension
\beq{redvd}
v:=g-1 + \int_{\beta} c_{1}(S) + n + h^{0,2}(S).
\eeq
\end{theorem}

\begin{proof}
By \cite{Beh, BF1} the relative obstruction theory for $\Mb_{\!g,n}(\S_B/B,\beta_B)$ is
\beq{relob}
E_{rel}\udot:= \left( R\pi_{*} R\hom\big(\big\{f^{*}\Omega_{\S_B/B} \to \Omega_{\mathcal C/\Mb}^{\mathrm{log}}\big\},\O_{\mathcal C}\big) \right)^{\vee}\To
\LL_{\Mb_{\!g,n}(\S_B/B,\beta_B)/B}\,,
\eeq
where the maps are those of the universal diagram
$$
\xymatrix@C=10pt@R=18pt{
\mathcal C \ar[r]^{f} \ar[d]_{\pi}  & \S_B \\
\Mb_{\!g,n}(\S_B/B,\beta_B).}
$$
By Proposition \ref{stablemapiso}, $\Mb_{\!g,n}(\S_B/B,\beta_B)\cong\Mb$ and the above diagram factors through the diagram \eqref{diag} for $S$ instead of $\S_B$. Therefore in fact $E\udot_{rel}$ is just $E\udot$ \eqref{Edot}, giving a perfect relative obstruction theory
\beq{Edot2}
E\udot\To\LL_{\Mb/B}.
\eeq
Even though $B$ is not smooth, its simple form -- and the fact that $\Mb$ is supported over the reduced point $0\in B$ -- means that there is an exact triangle
\beq{LL}
\Omega_B|_{\Mb}\To\LL_{\Mb}\To\LL_{\Mb/B}\To\Omega_B|_{\Mb}[1],
\eeq
where again we have suppressed the pullback map from $B$. Explicitly, embed $\Mb$ into an ambient stack $\A\to B$ which is smooth over $B$, and let $I$ denote the ideal of $\Mb\subset\A$. Then the triangle comes from the horizontal exact sequence of vertical complexes
$$
\spreaddiagramrows{-3mm}
\xymatrix{
& I/I^2 \ar[d]\ar@{=}[r] & I/I^2 \ar[d] \\ 
0\To\Omega_B|_{\Mb} \ar[r] & \Omega_\A|_{\Mb} \ar[r] & \Omega_{\A/B}|_{\Mb}\To0,\hspace{-15mm}}
$$
where the exactness of the bottom row follows from the smoothness of $\A\to B$.
From \eqref{Edot2} and \eqref{LL} we get a map $E\udot[-1]\to\Omega_B|_{\Mb}$ whose cone we define to be $F\udot$:
\beq{obthys}
\spreaddiagramrows{-.7pc}
\spreaddiagramcolumns{-.3pc}
\xymatrix{
F\udot \ar[d]\ar[r] & E\udot \ar[d]\ar[r] & \Omega_B|_{\Mb}[1] \ar@{=}[d] \\
\LL_{\Mb} \ar[r] & 
\LL_{\Mb/B} \ar[r] & \Omega_B|_{\Mb}[1].\!\!}
\eeq
To show that $F\udot=\mathrm{Cone}\big(E\udot[-1]\to\Omega_B|_{\Mb}\big)$ is quasi-isomorphic to $E_{red}\udot=\mathrm{Cone}\big(H^0(K_S)\otimes\O_{\Mb}[1]\to E\udot\big)$ it is sufficient to show that the composition
$$
H^0(K_S)\otimes\O_{\Mb}\To E\udot[-1]\To\Omega_B|_{\Mb}
$$
is an isomorphism.\footnote{The dual of this isomorphism says that moving in any direction in $B$ our curve is obstructed at first order since its cohomology class acquires a nonzero $(0,2)$ part $[\beta]^{0,2}$. It also proves the surjectivity claimed in \eqref{surjec}.} Applying $h^0$ we have
\beq{premess}
H^0(K_S)\otimes\O_{\Mb}\To R^0\pi_*\big(\big\{f^*\Omega_S\to\Omega_{\mathcal C/\Mb}^{\log}\big\}\otimes\omega_{\mathcal C/\Mb}\big)
\To\Omega_B|_{\Mb}.
\eeq
Recalling the definition of the dual semi-regularity map \eqref{reg} via the diagram \eqref{cxs} we see this factors through
\beq{mess}
H^0(K_S)\otimes\O_{\Mb}\To R^0\pi_*(f^*\Omega_S\otimes\omega_{\mathcal C/\Mb})
\To\Omega_B|_{\Mb},
\eeq
where the first arrow is the obvious map pulling back 2-forms to $C$. By Lemma \ref{KSclaim} below the second arrow -- induced by the Kodaira-Spencer map of $\Mb/B$ in the bottom row of \eqref{obthys} --  is the same as the one 
\beq{common}
R^0\pi_*(f^*\Omega_S\otimes\omega_{\mathcal C/\Mb})\To
R^1\pi_*(\Omega_B|_{\Mb}\otimes\omega_{\mathcal C/\Mb})\cong\Omega_B|_{\Mb}
\eeq
induced by the Kodaira-Spencer map $\Omega_S\to
\Omega_B|_{\Mb}[1]$ \eqref{cotangent} of $\S_B/B$. The composition is therefore the dual of \eqref{nondegen} pulled back to $\Mb$, which is indeed an isomorphism. \medskip

By the long exact sequences in cohomology of diagram \eqref{obthys}, and the fact that $E\udot$ is a perfect relative obstruction theory, we see that $F\udot$ has cohomology only in degrees $-1,\,0$ and that $F\udot\to\LL_{\Mb}$
is an isomorphism on $h^0$ and a surjection on $h^{-1}$. And since $E\udot$ is quasi-isomorphic to a 2-term complex of vector bundles $\{E^{-1}\to E^0\}$, so is $F\udot$ locally:
$$
F\udot\simeq\{E^{-1}\to E^0\oplus\Omega_B|_{\Mb}\}.
$$
(We work locally to obtain the map $E^{-1}\to\Omega_B$ from the fact that $E^{-1}$ is free, thus defining a projective module.) Working globally one can do the same thing by resolving by sufficiently negative locally free sheaves, etc.
Thus $F\udot\to\LL_{\Mb}$ is an absolute obstruction theory for $\Mb$.
\end{proof}

\begin{lemma} \label{KSclaim}
The second arrow in \eqref{mess} is the one induced by the Kodaira-Spencer map $\Omega_S\to\Omega_B|_{\Mb}[1]$ \eqref{cotangent} of $\S_B/B$.
\end{lemma}

\begin{proof}
Let $\MMb:=\MMb_{g,n}$ denote the stack of prestable curves with $n$ marked points, with universal curve $\mathcal C\to\MMb$. 

In passing from \eqref{premess} to \eqref{mess} we have passed from the obstruction theory $E\udot\to\LL_{\Mb/B}$ (composed with $\Mb$'s Kodaira-Spencer map $\LL_{\Mb/B}\to\Omega_B|_{\Mb}[1]$) to the relative obstruction theory $\mathcal E\udot\to\LL_{\Mb/\MMb\times B}$ (composed with $\Mb/\MMb$'s Kodaira-Spencer map $\LL_{\Mb/\MMb\times B}\to\Omega_B[1]$). In other words the second arrow in \eqref{mess} is the composition
$$
h^{-1}(\mathcal E\udot)\To h^{-1}(\LL_{\Mb/\MMb\times B})\To\Omega_B|_{\Mb}.
$$
We recall the relative obstruction theory of $\Mb/\MMb\times B$ \cite{Beh, BF1}, using the maps
\beq{CS}
\xymatrix@R=16pt{
\mathcal C\times_{\MMb}\Mb \ar[r]^(.6)f\ar[d]^\pi & \S_B \\
\Mb.\!}
\eeq
In fact it is simpler to describe the \emph{dual} of the perfect obstruction theory $\mathcal E\udot\to\LL_{\Mb/\MMb\times B}$; it is the composition $\LL_{\Mb/\MMb\times B}^\vee\to R\pi_*(f^*T_{\S_B/B})$ of the top row of the diagram of vertical exact triangles
\beq{bigdg}
\spreaddiagramcolumns{-.6pc}
\spreaddiagramrows{-.8pc}
\xymatrix{
\LL^\vee_{\Mb/\MMb\times B} \ar[r]\ar[d] & R\pi_*\!\Big(\!\pi^*
\LL^\vee_{\Mb/\MMb\times B}\oplus\LL^\vee_{\mathcal C/\MMb}\!\Big)\ar[d]\ar@{=}@<.5ex>[r]
& R\pi_*\LL^\vee_{\mathcal C\times_{\MMb}\Mb/\MMb\times B} \ar[r]^(.53){(f^*)^\vee}\ar[d] & R\pi_*f^*\LL^\vee_{\S_B/B} \ar[d] \\
\LL^\vee_{\Mb/\MMb} \ar[r]\ar[d] & R\pi_*\!\Big(\!\pi^*
\LL^\vee_{\Mb/\MMb}\oplus\LL^\vee_{\mathcal C/\MMb}\!\Big) \ar@{=}@<.5ex>[r]\ar[d]
& R\pi_*\LL^\vee_{\mathcal C\times_{\MMb}\Mb/\MMb} \ar[r]^(.53){(f^*)^\vee}\ar[d] &
R\pi_*f^*\LL^\vee_{\S_B} \ar[d] \\
V \ar[r] & R\pi_*(\pi^*V) \ar@{=}[r] & R\pi_*(\pi^*V) \ar@{=}[r] & R\pi_*(\pi^*V).}
\eeq
Here $\LL^\vee_{\S_B/B}$ is just the tangent bundle $T_{\S_B/B}$ since $\S_B\to B$ is smooth, $V$ is the splitting of \eqref{choice}, and we recall that $\Omega_B|_{\Mb}\cong V^*$.

Use the natural exact triangle of cohomologies
$$
V\To R\pi_*(\pi^*V)\To R^1\pi_*(\pi^*V)[-1]
$$
to remove $R^1\pi_*(\pi^*V)[-1]$ from the lower and the middle terms in the right hand column. In the latter case we call the resulting cone $(\mathcal F\udot)^\vee$. Now compose \eqref{bigdg} horizontally and dualise, to give a commutative diagram involving just the first and last columns:
$$
\spreaddiagramrows{-.6pc}
\xymatrix{
\LL_{\Mb/\MMb\times B} &  \mathcal E\udot \ar[l] \\
\LL_{\Mb/\MMb} \ar[u] & \mathcal F\udot \ar[u]\ar[l] \\
V^* \ar[u] & V^*.\!\! \ar[u]\ar[l]}
$$
Rearranging gives
$$
\spreaddiagramrows{-.6pc}
\spreaddiagramcolumns{-.3pc}
\xymatrix{
\mathcal F\udot \ar[d]\ar[r] & \mathcal E\udot \ar[d]\ar[r] & \Omega_B|_{\Mb}[1] \ar@{=}[d] \\
\LL_{\Mb/\MMb} \ar[r] & \LL_{\Mb/\MMb\times B} \ar[r] & \Omega_B|_{\Mb}[1],\!\!}
$$
just as in \eqref{obthys}, but for $\LL_{\Mb/\MMb}$ instead of $\LL_{\Mb}$.
But \eqref{bigdg} commutes, so the above diagram does too. The top right hand map is the one we're after (after taking $h^{-1}$), and from the last column of \eqref{bigdg} we see that it is indeed the one induced by the Kodaira-Spencer
map of $\S_B$ over $B$.
\end{proof}

If we work on $\Xb=\PP(K_S\oplus\O_S)$ instead of $S$ we get a similar semi-regularity map by replacing $K_S$ and $\Omega_S$ in \eqref{cxs} by
$\Omega^2_{\Xb}$ and $\Omega_{\Xb}$ respectively. Thus we remove $H^{1,3}(\Xb)\cong H^{0,2}(S)$ from the obstruction sheaf, and by very similar working obtain the following.

\begin{theorem} \label{reducedGWX}
Suppose that $\beta\in H^2(S,\Z)$ is a $(1,1)$ class inducing a surjection $\cup\beta\colon H^1(T_S)\to H^2(\O_S)$. Then there is a perfect obstruction theory $E_{red}\udot\to\LL_{\Mb}$ for $\Mb_{\!g,n}(\Xb,\iota_*\beta)$ with virtual dimension $h^{1,3}(\Xb)+n=h^{0,2}(S)+n$. \hfill$\square$
\end{theorem}

\subsection{Stable pairs}
We sketch how the arguments of the last Section are modified to prove the same result for stable pairs. In Appendix \ref{Dmitri} we use a different method to get a better result using only Condition \eqref{condition} in place of \eqref{condition2}.

As before we let $\Xb = \mathbb{P}(K_S \oplus \O_S)$ be the projective completion of the canonical bundle of $S$. Let $0 \neq \beta \in H_{2}(S,\Z)$ be of type $(1,1)$ and denote the inclusion of the zero-section by $\iota \colon S\into\Xb$.

Let $P:= P_{n}(\Xb, \iota_{*}\beta)$ be the moduli space of stable pairs on $\Xb$ with universal object $\II\udot = \{
\O_{\Xb\times P} \To \mathbb{F}\}$ over $\Xb\times P$. Using the projections
\beq{projns}
\spreaddiagramrows{-3mm}
\spreaddiagramcolumns{-5mm}
\xymatrix{
& \Xb\times P \ar[dr]^(.55){\pi_P}\ar[dl]_(.55){\pi_{\Xb}} \\
\Xb && P}
\eeq
and the relative dualising sheaf $\omega_{\pi_P}=\pi_{\Xb}^*\omega_{\Xb}$, the perfect obstruction theory for stable pair theory of $\Xb$ is \cite[Theorem~2.14]{PT1}
\begin{equation*}
E\udot:=R \pi_{P*} (R\hom(\II\udot,\II\udot)_{0} \otimes \omega_{\pi_P})[2] \To\LL_P.
\end{equation*}
Here $(\ \cdot\ )_{0}$ denotes trace-free part, and the virtual dimension is $\int_{\iota_{*} \beta} c_{1}(\Xb)=0$. This gives an obstruction sheaf\footnote[1]{We denote the $i$th cohomology sheaf of $R \pi_{P*} R\hom$ by $\ext_{\pi_P}^{i}$.}
\begin{equation*}
\mathrm{Ob}:= h^{1}((E\udot)^\vee) = \ext_{\pi_P}^{2}(\II\udot, \II\udot)_0.
\end{equation*}
Cupping with the Atiyah class $A(\II\udot)\in\Ext^1(\II\udot,\II\udot\otimes\LL_{\Xb\times P})$ and taking trace defines a semi-regularity map $(E\udot)^\vee\to H^{1,3}(\Xb)\otimes\O_P[-1]$ by the composition
\begin{align}
\ext_{\pi_P}^{2}(\II\udot, \II\udot)_{0} \subset\ &
\ext_{\pi_P}^{2}(\II\udot, \II\udot)
\spreaddiagramcolumns{7pt} \xymatrix{\ar[r]^{\cup A(\II\udot)} &}
\ext_{\pi_P}^{3}(\II\udot, \II\udot \otimes\LL_{\Xb\times P}) \To \nonumber \\
& \ext_{\pi_P}^{3}(\II\udot, \II\udot \otimes \pi_{\Xb}^{*} \Omega_{\Xb})
\Rt{\tr} R^3\pi_{P*} \pi_{\Xb}^{*} \Omega_{\Xb} \cong
H^{1,3}(\Xb)\otimes\O_P. \label{pairsreg}
\end{align}
We will see in the proof of Theorem \ref{reducedPT} that \eqref{pairsreg} is a surjection when $(S,\beta)$ satisfy Condition \eqref{condition2}. This also follows the obvious generalisation of \cite[Proposition~11]{MPT} to all surfaces.

Dualising the composition
$$
(E\udot)^\vee\To h^1((E\udot)^\vee)[-1]\To H^{1,3}(\Xb)\otimes\O_P[-1]
$$
gives a map
$$
H^{2,0}(\Xb)\otimes\O_P[1]\To E\udot;
$$
let $E\udot_{red}$ be its cone.

\begin{theorem} \label{reducedPT}
Assume that $\cup \beta\colon H^1(T_S)\to H^2(\O_S)$ is surjective.
Then there exists a perfect obstruction theory $E\udot_{red}\to\LL_P$ for $P_n(\Xb,\iota_*\beta)$ of virtual dimension $h^{1,3}(\Xb)=h^{0,2}(S)$.
\end{theorem}

\begin{proof}
Associated to the algebraic twistor family $\S_B\to B$ we get its family of projectively completed canonical bundles $\XXb_B\to B$. By Proposition \ref{stablepairiso} the family $P:= P_n(\XXb_B/B, \iota_{*}\beta_B)\to B$ of moduli spaces of stable pairs on the fibres is isomorphic to the space $P_n(\Xb,\beta)$ of stable pairs on $\Xb$.

In \cite[Section~3]{MPT}, a relative perfect obstruction theory $E_{rel}\udot$ is constructed
\begin{equation*}
E_{rel}\udot:= R \pi_{P*} (R\hom(\II\udot,\II\udot)_{0} \otimes \omega_{P \times_{B} \XXb_B / P })[2] \To \LL_{P/B}\udot.
\end{equation*}
Here $\XXb_B\times_B P$ carries the universal object $\II\udot$ and has a projection $\pi_P$ to $P$ with relative dualising sheaf $\omega_{\pi_P}=\pi_P^*\omega_{\XXb_B/B}$. As before, this can be made into a perfect absolute obstruction theory $F\udot$ by the diagram
$$
\spreaddiagramrows{-.5pc}
\xymatrix{
F\udot \ar[d] \ar[r] & E_{rel}\udot \ar[d]\ar[r] & \Omega_{B}[1] \ar@{=}[d] \\
\LL_{P} \ar[r] & \LL_{P/B} \ar[r] & \Omega_{B}[1].\!}
$$
Since by Proposition \ref{stablepairiso} the stable pairs of $P$ all lie scheme theoretically on the central fibre $\Xb$, we see as before that in fact $E\udot_{rel}$ is just the usual complex $E\udot$ of stable pair theory on $\Xb$. But $E\udot$ has virtual dimension $0$, so $F\udot$ has virtual dimension $h^{2,0}(S)$. Therefore to prove the Theorem we are left with showing that the composition
$$
F\udot\To E\udot\To E\udot_{red}
$$
is an isomorphism. It is sufficient to show that the composition
$$
H^0(\Omega^2_{\Xb})\otimes\O_P\To E\udot[-1]\To\Omega_B|_{\Mb}
$$
is an isomorphism. By the Nakayama lemma we may do so at a point $(F,s)\in P$. After dualising we get the map $V\to H^{1,3}(\Xb)$ given by the composition
\begin{align}
V\subset H^1(T_S)=\,H^1(T_{\Xb})&
\spreaddiagramcolumns{15pt}\xymatrix{\ar[r]^{\cup A(I\udot)} &}
\mathrm{Ext}^{2}(I\udot,I\udot)_{0} \ \subset\ \mathrm{Ext}^{2}(I\udot,I\udot) \nonumber \\
& \spreaddiagramcolumns{15pt}\xymatrix{\ar[r]^{\cup A(I\udot)} &}
\mathrm{Ext}^{3}(I\udot,I\udot \otimes \Omega_{\Xb})\Rt{\tr}H^{1,3}
(\Xb). \label{eqn1}
\end{align}
This uses the stable pairs analogue of Lemma \ref{KSclaim} (also proved in \cite[Proposition 13]{MPT}) to deduce that the composition of $E\udot\to\LL_{P/B}$ and the Kodaira-Spencer map $\LL_{P/B}\to\Omega_B[1]$ for $P$ coincides with the cup product of the Atiyah class and the Kodaira-Spencer class for $\Xb$.

In the proof of \cite[Proposition 11]{MPT}, it is observed that the above composition $H^{1}(T_{\Xb})\to H^{1,3}(\Xb)$ is equal to $\cup(-2\iota_*\beta)$. Thus on restriction to $V\subset H^1(T_S)$ it gives $-2$ times the isomorphism $\cup\beta$ of \eqref{nondegen}.
\end{proof}

\section{Invariants} \label{invts}
\subsection{Reduced Gromov-Witten invariants}

The reduced obstruction theory of Theorem \ref{reducedGW} gives, by \cite{BF1}, a virtual fundamental class which we call the \emph{reduced class}:
$$
[\Mb_{\!g,n}(S,\beta)]^{red}\in H_{2v}(\Mb_{\!g,n}(S,\beta)), \qquad
v=g-1 + \int_{\beta} c_{1}(S) + n + h^{0,2}(S).
$$
Integrating insertion cohomology classes over this gives the \emph{reduced Gromov-Witten invariants} of $S$. Namely, if $\sigma_i\in H^*(S,\Z)$ are cohomology classes, then\footnote{The $\sigma_i$ can be repeated, so for instance $R_{g,\beta}(S,\sigma_1^2\sigma_2)$ denotes $R_{g,\beta}(S,\sigma_1\sigma_1\sigma_2)$.}
\beq{Rgbeta}
R_{g,\beta}(S,\sigma_1\ldots\sigma_n):=
\int_{[\Mb_{\!g,n}(S,\beta)]^{red}}\prod_{i=1}^n\ev^*_i(\sigma_i)\ \in\ \Q.
\eeq
Here $\ev_i$ is the evaluation map from the $i$th marked point of the universal curve to $S$. So for a surface in the Noether-Lefschetz locus for $\beta$, the invariants give a virtual count of the curves in homology class $\beta$ which intersect $PD(\sigma_i)$.

\begin{remark}\textbf{\emph{Deformation invariance.}}
By standard theory \cite[Section 7]{BF1}, the
$R_{g,\beta}(\sigma_1\ldots\sigma_n)$ are invariant under deformations of $S$ within the Noether-Lefschetz locus. The usual arguments apply: given a smooth curve $Z$ mapping to the Noether-Lefschetz locus for $\beta$, we can make all of the constructions of the previous sections relative to $Z$ in the family over $Z$. (We do not even need to change notation; we can work with affine $Z$ and just let our ground ring be $\O_Z$ instead of $\C$.) The resulting obstruction theory is relative to $Z$, and restricts to the absolute obstruction theory of the previous section over any point of $Z$. As a result the relative virtual cycle on the relative moduli space over $Z$ pulls back, via the usual Gysin maps, to the virtual cycle on any fibre \cite[Proposition 7.2]{BF1}. The cohomology classes $\ev^*_i(\sigma^i)$ are defined on the relative moduli space, so by conservation of number \cite[Theorem 10.2]{Ful}, their integrals over a fibre of the virtual cycle is independent of the fibre. The same applies to the other invariants we define below.
\end{remark}
\medskip

In the usual way we can also define the same invariants \eqref{Rgbeta} without using marked points. Instead we use the universal map $f\colon\mathcal C\to S$ from the universal curve $\pi\colon\mathcal C\to\Mb_{\!g}(S,\beta):=
\Mb_{\!g,0}(S,\beta)$.
Then we claim that
\beq{Rgbeta2}
R_{g,\beta}(S,\sigma_1\ldots\sigma_n)=
\int_{[\Mb_{\!g}(S,\beta)]^{red}}\prod_{i=1}^n\pi_*f^*(\sigma_i).
\eeq
In fact we can remove one marked point at a time using the diagram
$$
\spreaddiagramrows{-6mm}
\spreaddiagramcolumns{-7mm}
\xymatrix{
\Mb_{\!g,n}(S,\beta) \ar[rrd]^{\ev_n}\ar[dr]_(.7)\rho\ar[dddr]_r \\
& \mathcal C \ar[r]_f\ar[dd]^{\pi} & S \\ \\
& \Mb_{\!g,n-1}(S,\beta).}
$$
The map $r$ forgets the $n$th marked point and stabilises the resulting curve and map, while $\rho$ maps the $n$th point to its image in the contracted curve. Since $\rho$ is birational, we find that
\beq{key}
r_*\ev_n^*(\sigma)=r_*\rho^*f^*(\sigma)=\pi_*f^*(\sigma).
\eeq
Iterating we can push all the way down from $\Mb_{\!g,n}$ to $\Mb_{\!g}$.
The compatibility of the ordinary obstruction theories of $\Mb_{\!g,n}$ and $\Mb_{\!g,n-1}$ is \cite[Axiom IV]{Beh}. For the reduced theories the same argument applies because they are defined by the semi-regularity map \eqref{reg} which is compatible with $r$: its construction \eqref{cxs} does not even see the marked points. The equality of \eqref{Rgbeta} and \eqref{Rgbeta2} follows. \medskip

Since $\Mb_{\!g,n}(X,\iota_*\beta)\subset\Mb_{\!g,n}(\Xb,\iota_*\beta)$ and
$P_n(X,\iota_*\beta)\subset P_n(\Xb,\iota_*\beta)$ are open immersions, they inherit the reduced obstruction theories of Theorems \ref{reducedGWX} and \ref{reducedPT} by restriction. But they are noncompact, so to define invariants we have to use residues and the virtual localisation formula. $T:=\C^*$ acts with weight one on the fibres of $X=K_S$ with fixed locus $S$. Therefore it acts on the moduli spaces
$\Mb_{\!g,n}(X,\iota_*\beta)$ and $P_n(X,\iota_*\beta)$. Its fixed loci are related to the curves in the zero-section $S$. For stable maps we get precisely $\Mb_{\!g,n}(S,\beta)$:

\begin{proposition} \label{GWloc}
The inclusion $\Mb_S\into\Mb_X^{T}$ is an isomorphism of stacks. Moreover $E\udot_{X,red}$ is naturally $T$-equivariant and its restriction to $\Mb_S$ has fixed and moving parts
\beqa
\left(E\udot_{X,red}|_{\Mb_X^{T}}\right)^{fix} &\cong& E\udot_{S,red}, \\
\left(E\udot_{X,red}|_{\Mb_X^{T}}\right)^{mov} &\cong& (R \pi_{*} f^{*} K_S \otimes\mathfrak t)^\vee,
\eeqa
where $\mathfrak t$ is the irreducible representation of weight $1$.
\end{proposition}
\begin{proof}
The isomorphism
$$
T_X|_S\cong T_S\oplus K_S
$$
induces an isomorphism on $\Mb_S$,
\beq{fixobs}
E\udot_X|_{\Mb_S}\cong E_S\udot\oplus\big(R\pi_*(f^*K_S)\big)^\vee.
\eeq
The first summand carries the trivial $T$-action, the second carries the weight-$(-1)$ action induced from the action on the fibres of $K_S$. 

We want to show that the inclusion $\Mb_S\into\Mb_X^{T}$ is an isomorphism of stacks. It is sufficient to show that it induces an isomorphism on maps from $\Spec A_n$ to the moduli space, where $A_n$ is any Artinian $\C$-algebra of length $n$. The $n=0$ case is the obvious fact that $\Mb_S\into\Mb_X^{T}$ is a bijection of sets.

Inductively we fix a surjection $A_{n+1}\to A_n$ with ideal $I$, and a map
$$
a\colon\Spec A_n\to\Mb_S\into\Mb_X^{T}.
$$
We show that any lift to a map $\Spec A_{n+1}\to\Mb_X^{T}$ factors through $\Mb_S$.

By \cite[Theorem 4.5]{BF1} such a lift exists if and only if the obstruction class in $\Ext^1(a^*(E\udot_X|_{\Mb_X^T})^{fix},I)$ vanishes. (Here we have used the fact that the $T$-fixed part $(E\udot_X|_{\Mb_X^T})^{fix}$ of $E\udot_X|_{\Mb_X^T}$ provides an obstruction theory for $\Mb_X^T$ \cite{GP}.)

By the isomorphism \eqref{fixobs} this is the same as the obstruction in $\Ext^1(a^*E_S\udot,I)$ of finding a lift to $\Mb_S$. So if a lift to $\Mb_X^T$ exists, so does one to $\Mb_S$. By \cite[Theorem 4.5]{BF1} and \eqref{fixobs} the choices in such a lift are also the same
$$
\Hom(a^*(E\udot_X|_{\Mb_X^T})^{fix},I)\ \cong\ \Hom(a^*E_S\udot,I).
$$
It follows that the lifts that factor through $\Mb_S$ map isomorphically to the lifts to $\Mb_X^T$, as required.

Finally, by their very constructions, the semi-regularity maps of $\Xb,\,S$ intertwine the isomorphism \eqref{fixobs}:
$$
\xymatrix@=18pt{
H^{3,1}(\Xb)\otimes\O_{\Mb_S}[1] \ar@{=}[r]\ar[d] &
H^{2,0}(S)\otimes\O_{\Mb_S}[1] \ar[d] \\
E\udot_{\Xb}|_{\Mb_S} \ar@{=}[r] & E_S\udot\! &\hspace{-2cm}\oplus\,\big(R\pi_*(f^*K_S)\big)^\vee.}
$$
Taking cones gives the isomorphisms
$$
E_{X,red}\udot|_{\Mb_S}\cong E_{S,red}\udot\,\oplus\big(R\pi_*(f^*K_S)\big)^\vee
$$
over $\Mb_S\cong\Mb_X^T$.
\end{proof}

Therefore we can define reduced Gromov-Witten residue invariants of $X$ using
Graber-Pandharipande's virtual localisation formula \cite{GP}. 
That is, via $\big(E\udot_{X,red}|_{\Mb_X^{T}}\big)^{fix}$ we get a perfect obstruction theory for $\Mb_X^{T}$ and so a virtual cycle $[\Mb_X^{T}]^{red}$. Then, given equivariant cohomology classes $A_i\in H^*_T(\Mb_X)$, we define
$$
\int_{[\Mb_X]^{red}} \prod_iA_i\ :=\,\int_{[\Mb_X^{T}]^{red}}
\frac{1}{e(N^{vir})} \prod_iA_i\ \in\Q(t).
$$
Here $t=c_1(\mathfrak t)$ is the equivariant parameter -- the generator of $H^*(BT)=\Q[t]$ -- and the virtual normal bundle is defined to be $\big(E\udot_{X,red}
|_{\Mb_X^{T}}\big)^{\!\vee\,mov}$. Expressing this as a two-term complex $E_0\to E_1$ of equivariant bundles \emph{whose weights are all nonzero}\footnote{This is possible, and ensures that the $c_{top}(E_i)$ are invertible in the localised equivariant cohomology ring. Then $e(N^{vir})$ is independent of the choice of resolution $E_0\to E_1$.} its virtual equivariant Euler class is defined to be
$$
e(N^{vir}):=c_{top}(E_0)/c_{top}(E_1)\in H^*_T(\Mb_X^T)\otimes_{\Q[t]}\Q(t), $$
where $c_{top}$ is the $T$-equivariant Chern class. By Proposition \ref{GWloc}, this gives
$$
\int_{[\Mb_S]^{red}}
\frac{1}{e(R \pi_{*} f^{*} K_S \otimes\mathfrak t)} \prod_iA_i.
$$
In particular we can define
\beq{RgbetaX}
\curly R_{g,\beta}(X,\sigma_1\ldots\sigma_n):=
\int_{[\Mb_{\!g,n}(S,\beta)]^{red}}\frac{1}{e(R \pi_{*} f^{*} K_S \otimes \mathfrak t)}
\prod_{i=1}^n\ev^*_i(\sigma_i)\!\!
\eeq
in $\Q(t)$; compare \eqref{Rgbeta}. (Throughout we use curly letters to emphasise residue invariants in $\Z(t)$ or $\Q(t)$; straight letters denote numerical invariants in $\Z$ or $\Q$.)

Setting $r:=-\rk(R \pi_{*} f^{*} K_S)=-\chi(f^*K_S)=g-1+\int_\beta c_1(S)$, we have $1/e(R \pi_{*} f^{*} K_S \otimes t)=t^r+O(t^{r-1})$. In particular, we find

\begin{lemma} \label{X=S}
The leading coefficient of the reduced Gromov-Witten invariants of $X$ \eqref{RgbetaX} reproduces the reduced Gromov-Witten invariants of $S$ \eqref{Rgbeta}:
$$
\curly R_{g,\beta}(X,\sigma_1\ldots\sigma_n)\ =\ 
R_{g,\beta}(S,\sigma_1\ldots\sigma_n)\,t^r\ +\ O(t^{r-1}),
$$
where $r=g-1+\int_\beta c_1(S).\hfill\square$
\end{lemma}

Note our controversial use of the term ``leading coefficient": it is possible for this be zero but that the whole polynomial $\curly R_{g,\beta}\ne0$.

\subsection{Reduced stable pair invariants}
The reduced obstruction theory of Theorem \ref{reducedPT} restricts from $P_n(\Xb, \iota_{*}\beta)$ to endow the open set $P_n(X, \iota_{*}\beta)$ with a perfect obstruction theory $E\udot_{X,red}$. The action of $T=\C^*$ on the fibres of $X=K_S$ defines a $T$-action on $P_{n}(X,\iota_{*}\beta)$ with respect to which $E\udot_{X,red}$ is $T$-equivariant. We will define stable pair invariants of $X$ using residues and the virtual localisation formula.

As usual let $\II\udot:=\{\O\to\mathbb F\}$ denote the universal complex over $X\times P_n(X,\iota_*\beta)$. The universal curve (the scheme-theoretic support of $\mathbb F$) represents $ch_2(\mathbb F)$. Using the usual projections \eqref{projns}, we define the following cohomology class for each $\sigma_i\in H^*(X,\Z)$
$$
\tau(\sigma_i):=\pi_{P*}\big(ch_2(\mathbb F) \cdot \pi_X^*(\sigma_i)\big) \in H^*(P_{n}(X,\iota_*\beta),\Z).
$$
For nonzero $\beta\in H_2(X,\Z)$, we would like to
define the stable pair invariant with insertions by
$$
\curly P^{red}_{\!n,\beta}(X,\sigma_1\ldots\sigma_m)\ :=\
\int_{[P_{n}(X,\beta)]^{red}}
\left(\prod_{i=1}^m \tau(\sigma_i)\right).
$$
We make sense of this as a residue by the virtual localisation formula \cite{GP}:
\beq{pairsdefn}
\curly P^{red}_{\!n,\beta}(X,\sigma_1\ldots\sigma_m)\ :=\
\int_{[P_{n}(X,\beta)^T]^{red}}\frac1{e(N^{vir})}
\left(\prod_{i=1}^m \tau(\sigma_i)\right)\in\Z(t).
\eeq
In contrast to the Gromov-Witten case, the fixed point locus can contain pairs that are supported set theoretically but not scheme theoretically on $S$. However, we next check that $P_S = P_{n}(S,\beta)$ does provide one connected component of the fixed locus. Therefore the invariants \eqref{pairsdefn} have a contribution coming entirely from $S$.

Over $P_S\subset P_X$ we slightly modify our usual notation and let $\iota_*\mathbb F$ denote the universal sheaf, where $\iota\colon S\into X$ is the inclusion of the zero-section. Then we have two universal complexes,
$$
\II\udot_S:=\{\O_{S\times P_S}\to\mathbb F\} \quad\mathrm{on\ }S\times P_S,
$$
and
$$
\II\udot_X:=\{\O_{X\times P_S}\to\iota_*\mathbb F\} \quad\mathrm{on\ }X\times P_S.
$$

\begin{proposition} \label{PTloc}
The subscheme $P_{n}(S,\beta)\subset P_{n}(X,\beta)^T$ is both open and closed in the fixed locus. On this component, the obstruction theory $E_X\udot$ has fixed part
\beq{pafix}
(E_X\udot|_{P_S})^{fix}\ \cong\ (R\pi_{P*}R\hom(\II\udot_S,\mathbb F))^\vee,
\eeq
and moving part its shifted dual
\beq{pamov}
(E_X\udot|_{P_S})^{mov}\ \cong\ 
R\pi_{P *} R\hom(\II\udot_S, \mathbb{F})[1] \otimes\mathfrak t^*
\eeq
twisted by the irreducible representation $\mathfrak t^*$ of weight $-1$.

Moreover, the reduced obstruction theory $E\udot_{X, red}|_{P_S}$ has the same moving part $(E\udot_X|_{P_S})^{mov}$ and
$$
(E_{X,red}\udot|_{P_S})^{\vee fix}\ \cong\ \mathrm{Cone}\big(
R\pi_{P*}R\hom(\II\udot_S,\mathbb F)\To H^{0,2}(S)\otimes\O_{P_S}[-1]\big),
$$
where the map is obtained as the composition
$$
R\pi_{P*}R\hom(\II\udot_S,\mathbb F)\To R\pi_{P*}R\hom(\mathbb F,\mathbb F)[1]\Rt{\tr}R\pi_{P*}\O[1]\Rt{\tau^{\ge1}}R^2\pi_{P*}\O[-1].
$$
\end{proposition}

\begin{proof}
The triangle $\II\udot_X\to\O_{X\times P_S}\to\iota_*\mathbb F$ gives
the following commutative diagram of exact triangles over $P_S$:
$$
\spreaddiagramcolumns{-10pt}
\spreaddiagramrows{-6pt}
\xymatrix{
& R\pi_{P*}\O[1] \ar@{=}[r]\ar[d]^\id & R\pi_{P*}\O[1] \ar[d] \\
\!\!R\pi_{P*}R\hom(\II\udot_X,\iota_{*} \mathbb F) \ar[r]&
R\pi_{P*}R\hom(\II\udot_X,\II\udot_X)[1] \ar[r]\ar[d]&
R\pi_{P*}R\hom(\II\udot_X,\O)[1] \ar[d] \\
& R\pi_{P*}R\hom(\II\udot_X,\II\udot_X)_0[1] \ar[r]&
R\pi_{P*}R\hom(\iota_*\mathbb F,\O)[2].}
$$
By Serre duality down $\pi_P\colon X\times P_S\to P_S$, the last term is $$
(R\pi_{P*}(\iota_*\mathbb F\otimes\omega_{\pi_P}))^\vee[-1].
$$
But $\iota_*\mathbb F$ is fixed by $T$, while $\omega_{\pi_P}$ is just the pullback of $K_X$, which is trivial but with $T$-weight $-1$. Therefore taking fixed parts removes this term and gives the isomorphism\footnote{
$(R\pi_{P*}R\hom(\II\udot_X,\iota_*\mathbb F))^\vee$ provides the natural obstruction theory for the moduli space of stable pairs $(F,s)$. This is essentially proved in \cite{Ill} once one combines it with \cite[Theorem 4.5]{BF1}: see \cite[Sections 12.3-12.5]{JS} for a full account. However it is \emph{not} perfect in general, and to define stable pair invariants one uses instead $(R\pi_{P*}R\hom(\II\udot_X,\II\udot_X)_0[1])^\vee$ \cite{PT1}.
The two theories give the same tangents, but different obstructions. Here we see that they become the same on $S\subset K_S$ once we pass to fixed parts.\label{pairsfoot}}
\beq{need}
\big(R\pi_{P*}R\hom(\II\udot_X,\iota_{*} \mathbb F)\big)^{fix}\ \cong\ 
\big(R\pi_{P*}R\hom(\II\udot_X,\II\udot_X)_0\big)^{fix}[1]
\eeq
over $P_S\subset P_X^T$.

Following \cite[Appendix~C]{PT3}, we next consider the exact triangle
\begin{equation*}
\mathbb F\otimes N_{S/X}^* \To L \iota^{*}\II\udot_X \To \II\udot_S,
\end{equation*}
where now $\iota\colon S\times P_S\into X\times P_S$.
Applying $R\pi_{P*}R\hom(\ \cdot\ ,\mathbb F)$ gives the exact triangle
$$
R\pi_{P*}R\hom(\II\udot_S,\mathbb F) \To
R\pi_{P*}R\hom(\II\udot_X,\iota_{*} \mathbb F) \To
R\pi_{P*}R\hom(\mathbb F,\mathbb F \otimes K_S).
$$
The first term has $T$-weight 0; the last has $T$-weight 1. Taking fixed parts,
\beq{pfix}
\big(R\pi_{P*}R\hom(\II\udot_X,\iota_{*} \mathbb F)\big)^{fix}\ \cong\ 
R\pi_{P*}R\hom(\II\udot_S,\mathbb F).
\eeq
Combined with \eqref{need} this gives \eqref{pafix}.

By \cite{GP}, the left hand side of \eqref{pafix} defines a perfect obstruction theory for $P_X^T$, while the right hand side defines one\footnote{This obstruction theory -- the surface analogue of that in footnote \ref{pairsfoot} -- \emph{is} perfect. This is shown in \cite{Ott} using the fact that $\mathbb F$ has relative dimension 1 support over $P_S$, so
$R^{\ge2}\pi_{P*}\mathbb F=0$, which combines with the exact triangle $\mathbb F[-1]\to\II\udot_S\to\O$ to force $\ext^{\ge2}_{\pi_P}(\II\udot_S,\mathbb F)$ to vanish. The results of \cite{GP} together with \eqref{pafix} give a different proof of this fact. \label{OTT}} for $P_S$. Since they are isomorphic over $P_S\subset P_X^T$, the proof of Proposition \ref{GWloc} now shows that $P_S\into P_X^T$ is a local isomorphism, as claimed.

To derive \eqref{pamov} we use the Serre duality
$$
R\pi_{P*}R\hom(\II\udot_X,\II\udot_X)_0[1]\ \cong\ 
(R\pi_{P*}R\hom(\II\udot_X,\II\udot_X\otimes\omega_{\pi_P})_0)^\vee[-2].
$$
But $\omega_{\pi_P}\cong\O\otimes\mathfrak t^*$, so this says that
\beq{nearly}
(E_X\udot)^\vee\cong E_X\udot[-1]\otimes\mathfrak t
\eeq
on restriction to $P_S\subset P_X$. We have already seen above that $E_X\udot|_{P_S}$ has $T$-weights only $0$ and $-1$, so tensoring with $\mathfrak t$ makes moving parts fixed and vice-versa. Taking fixed parts of \eqref{nearly} therefore gives \eqref{pamov}. \medskip

Finally we have to identify the semi-regularity map on the obstruction theory. Since it is $T$-equivariant, it is only nonzero on the fixed part. Recall\footnote{From now on we work on the compactification $\Xb$, so $\II\udot$ denotes the complex $\{\O_{\Xb \times P_S} \to \iota_* \mathbb F\}$.} its definition \eqref{pairsreg} by cupping with the Atiyah class $A(\II\udot)$ of $\II\udot$ and taking trace. By naturality of the Atiyah class \cite[Proposition 3.11]{BuF}, the four left hand squares in the following diagram commute. The right hand square commutes because $\tr(a\circ b)=tr(b\circ a)$. The unmarked arrows are all induced by the connecting homomorphism
$\iota_*\mathbb F[-1]\to\II\udot$.
$$ \hspace{6mm}
\xymatrix@C=35pt@R=20pt{
\ext^2_{\pi_P}(\II\udot,\II\udot) \ar[r]^(.42){\circ A(\II\udot)} & \ext^3_{\pi_P}(\II\udot,\II\udot\otimes\LL_{\Xb\times P})
\ar[r]^(.65){\tr_{\II\udot}} & R^3\pi_{P*}\Omega_{\Xb} \ar@{=}[ddd] \\
\hspace{-2cm}\ext^1_{\pi_P}(\II\udot,\iota_*\mathbb F) \ar[r]^(.35){\circ A(\iota_*\mathbb F)}\ar@{=}[d]<-5ex>\ar<-20pt,-30pt>;<0pt,-10pt> & \ext^2_{\pi_P}(\II\udot,\iota_*\mathbb F\otimes\LL_{\Xb\times P}) \ar[u]\ar@{=}[d] \\
\hspace{-2cm}\ext^1_{\pi_P}(\II\udot,\iota_*\mathbb F) \ar[r]^(.35){A(\II\udot)\circ}
\ar<-20pt,-90pt>;<0pt,-110pt> &
\ext^2_{\pi_P}(\II\udot,\iota_*\mathbb F\otimes\LL_{\Xb\times P}) \ar[d] \\
\ext^2_{\pi_P}(\iota_*\mathbb F,\iota_*\mathbb F) \ar[r]^(.42){A(\iota_*\mathbb F)\circ} &
\ext^3_{\pi_P}(\iota_*\mathbb F,\iota_*\mathbb F\otimes\LL_{\Xb\times P})
\ar[r]^(.65){\tr_{\iota_*\mathbb F}} & R^3\pi_{P*}\Omega_{\Xb}.\!}
$$
Our semi-regularity map starts with the fixed part of $\ext^1_{\pi_P}
(\II\udot,\iota_*\mathbb F)$ on the left, takes it clockwise round the diagram to $R^3\pi_{P*}\Omega_{\Xb}$. Therefore this is the same as going anticlockwise, via $\ext^2_{\pi_P}(\iota_*\mathbb F,\iota_*\mathbb F)$. By adjunction and the isomorphism $L\iota^*\iota_*\cong
\id\oplus(\id\otimes K_S)[1]$ this is
\beq{tube}
\ext^2_{\pi_P}(\iota_*\mathbb F,\iota_*\mathbb F)\ \cong\
\ext^2_{\pi_P}(\mathbb F,\mathbb F)\oplus \ext^1_{\pi_P}(\mathbb F,\mathbb F\otimes K_S).
\eeq
We are only interested in the $T$-fixed part, i.e. the first summand above.

Now, $\iota_*\mathbb F\cong q^*\mathbb F\otimes\iota_*\O_S$, where $q\colon\Xb\to S$ is the projection, and we have omitted the pullback maps along $P$. Therefore
\beq{tube2}
A(\iota_*\mathbb F)=q^*A(\mathbb F)\otimes1_{\iota_*\O_S}+
1_{\mathbb F}\otimes A(\iota_*\O_S).
\eeq
The first summand acts trivially on the first summand of \eqref{tube} (since $\ext^3_{\pi_P}(\mathbb F,\mathbb F \break \otimes\Omega_{\Xb}|_S)=0$). For the second summand, $A(\iota_*\O_S)$ lies in
$$
H^1\big(R\hom(\iota_*\O_S,\iota_*\O_S)\otimes\Omega_{\Xb}\big)
\ =\ H^1\big(\iota_*\O_S\otimes\Omega_{\Xb}\big)\ \oplus\ H^0\big(\iota_*\O_S(S)
\otimes\Omega_{\Xb}\big)
$$
and is the canonical element of the second summand: the section $\tau$ of $T_{\Xb}^*
\otimes\O_S(S)$ that projects tangent vectors to $\Xb$ to the normal bundle of $S\subset\Xb$.

So applying the second summand of \eqref{tube2} to the first of \eqref{tube} and then taking trace gives the upper composition in the commutative diagram
$$
\xymatrix@C=18pt@R=10pt{
& \ext^2_{\pi_P}(\mathbb F,\mathbb F\otimes\Omega_{\Xb}|_S(S)) \ar[rd]^(.55){\tr} \\ \ext^2_{\pi_P}(\mathbb F,\mathbb F) \ar[ru]^(.45)\tau\ar[rd]^(.55){\tr} && R^2\pi_{P*}\Omega_{\Xb}|_S(S) \ar[r] & R^3\pi_{P*}\Omega_{\Xb}. \\
& R^2\pi_{P*}\O \ar[ru]^(.4)\tau}
$$
The right hand map is most easily defined by duality: it is Serre dual to the composition $R^0\pi_{P*}\Omega_{\Xb}^2\to R^0\pi_{P*}(\Omega_{\Xb}|_S\otimes
\Omega_{\Xb}|_S)\to R^0\pi_{P*}(\Omega_{\Xb}|_S(-S))$. Thus the composition
$R^2\pi_{P*}\O\to R^3\pi_{P*}\Omega_{\Xb}$ is an isomorphism: the pullback to $P$ of the isomorphism $\iota_*\colon H^{0,2}(S)\to H^{1,3}(\Xb)$.
Therefore the commutativity of this diagram proves the last claim of the Proposition.
\end{proof}

In particular, $P_S=P_n(S,\beta)$ carries a reduced perfect obstruction theory
$(E_{X,red}\udot|_{P_S})^{fix}$ and a corresponding reduced virtual cycle
\beq{Predvd}
[P_S]^{red}\in H_{2v}(P_S), \qquad v:=2h-2 + n + \int_{\beta} c_{1}(S) + h^{0,2}(S),
\eeq
of virtual dimension $v$. Here $v - h^{0,2}(S) = \rk(R\pi_{P *} R\hom(\II\udot_S, \mathbb{F}))=2h-2+n+\int_\beta c_1(S)$. Thus we can define the reduced invariants of $S$ to be
\beq{SPS}
P^{red}_{n,\beta}(S,\sigma_1\ldots\sigma_m)\ :=\ 
\int_{[P_{n}(S,\beta)]^{red}}\left(\prod_{i=1}^m\tau(\sigma_i)\right)\ \in\ \Z.
\eeq
Secondly we can use the virtual localisation formula to define the reduced \emph{residue} invariants of $S$ to be the contribution of the component $P_S\subset P_X^T$ to the stable pair invariants of $X$ \eqref{pairsdefn}. This is
\beq{Pres}
\curly P^{red}_{\!n,\beta}(S,\sigma_1\ldots\sigma_m)\ :=\
\int_{[P_{n}(S,\beta)]^{red}}\frac1{e(N^{vir})}
\left(\prod_{i=1}^m \tau(\sigma_i)\right)\ \in\,\Z(t),
\eeq
where $N^{vir}=(R\pi_{P *} R\hom(\II\udot_S, \mathbb{F}))^\vee[-1] \otimes \mathfrak t$. Thus $1/e(N^{vir})=t^r+O(t^{r-1})$, where $r:=\rk(R\pi_{P *} R\hom(\II\udot_S, \mathbb{F}))= v-h^{0,2}(S)$, and
$$
\curly P^{red}_{\!n,\beta}(S,\sigma_1\ldots\sigma_m)\ =\,\big(
P^{red}_{n,\beta}(S,\sigma_1\ldots\sigma_m)\big)t^r+O(t^{r-1}).
$$
That is the residue invariants contain as their leading coefficient the reduced stable pair invariants of $S$; cf. Lemma \ref{X=S}. Of course it is often the case (for degree reasons, for instance) that the latter vanishes while the former does not. We will see an example of this in Section \ref{nodal}.

Under certain circumstances, $P_S$ is all of $P_X^T$, so that
$\curly P^{red}_{\!n,\beta}(S)=\curly P^{red}_{\!n,\beta}(X)$. The following proposition, proved in the sequel \cite{KT2}, gives examples of this. We will not use this result in the current paper. 

\begin{proposition} [{\cite[Proposition 5.1]{KT2}}] \label{irreducible}
In the following two cases there is an isomorphism $P_{n}(X,\iota_{*}\beta)^{T} \cong P_{n}(S,\beta)$ 
\begin{itemize}
\item $\beta$ is irreducible, or
\item $K_S^{-1}$ is nef, $\beta$ is $(2\delta+1)$-very ample\footnote{By this we mean there exists a line bundle in $\Pic_\beta(S)$ which is $(2\delta+1)$-very ample. Recall \cite{BS} that this means that $H^0(L)\to H^0(L|_Z)$ is surjective for every length $2\delta+2$ subscheme $Z$ of $S$.} and $n \leq 1 - h + \delta$.
\end{itemize}
(Here $h$ is the arithmetic genus of curves in class $\beta$, determined by $2h-2=\beta^2-c_1(S).\beta$. The inequality on $n$ means the stable pairs have $\le\delta$ free points.)
\end{proposition}

\noindent The previous proposition is false for arbitrary surfaces. For instance if $K_S=\O_S(C_0)$ is effective, then consider $\beta=nC_0$ and let $C$ be the $n$-fold thickening of $C_0$ along the fibres of $K_S$. This is $T$-fixed with $\chi = 1 - h$, but not scheme theoretically supported on $S$. However one can often make it true again by restricting to small linear subsystems in the space of curves. We will do this in Section \ref{nodal}.

\section{Insertions and linear systems} \label{insert}
\subsection{Det and div}

Let $\Hilb_\beta(S)$ denote the Hilbert scheme of curves\footnote{These are subschemes $Z\subset S$ of Hilbert polynomial $\chi(\O_Z(n))=\frac12\int_\beta c_1(S)-\beta^2/2+n\int_\beta c_1(\O(1))$ for every ample line bundle $\O(1)$ on $S$. In contrast to the threefold case, these are all pure curves, i.e. subschemes of pure dimension one, with no free or embedded points.} in $S$ in class $\beta$. Such a curve $C$ is a divisor with an associated line bundle $\O(C)$, defining an Abel-Jacobi map
$AJ\colon\Hilb_\beta(S)\to\Pic_\beta(S)$. Both of these spaces receive maps from both of our moduli spaces of curves in $S$:
\beq{maps}
\spreaddiagramcolumns{-7pt}
\spreaddiagramrows{-6pt}
\xymatrix{
\Mb_{\!g,n}(S,\beta)\!\! \ar[rd]^\div\ar[rdd]_\det && P_n(S,\beta) \ar[ld]_H\ar[ldd]^\det
\\ & \Hilb_\beta(S) \ar[d]^(.4){\hspace{-2.5pt}AJ} \\ & \Pic_\beta(S).}
\eeq
The map div takes a stable map to its \emph{divisor class}, which is its image with multiplicities. That is, if the irreducible components $C_i$ of its image are multiply covered $k_i$ times then the image is the divisor $\sum_ik_iC_i$ defined by the ideal sheaf
$\bigotimes_i\I_{C_i}^{k_i}$. This set theoretic map can be made into a morphism by taking a stable map $f\colon C\to S$ to the line bundle
$\det(f_*\O_C)\in\Pic_\beta(S)$ and its canonical section \cite{KM}.

The map $H$ takes a pair $(F,s)$ to the scheme theoretic support of $F$. In fact it is proved in \cite[Proposition B.8]{PT3} that $P_n(S,\beta)$ is a relative Hilbert scheme of points on the fibres of the universal curve over $\Hilb_\beta(S)$.

The fibre of $AJ$ over the line bundle $L$ is the full linear system $\PP(H^0(L))$. If we wish to derive invariants from just one such linear system, we can do so using insertions as in \cite{BL2}. We reprove their results in a slightly simpler way.

Picking a basepoint in $\Pic:=\Pic_\beta(S)$ gives a canonical isomorphism $\Pic\cong H^1(S,\R)/H^1(S,\Z)$. Therefore $H^1(\Pic,\R)$ is canonically isomorphic to the space $H^1(S,\R)^*$ of constant 1-forms on $\Pic$. Via the isomorphism $H_1(S,\R)\cong H^1(S,\R)^*$, a cycle $\gamma\in H_1(S)$ gets taken to the constant 1-form $\tilde\gamma$ whose pairing with a constant tangent vector $v\in H^1(S,\R)$ to $\Pic$ is
$$
\langle\tilde\gamma,v\rangle_{\Pic}=\int_\gamma v,
$$
where the right hand integral takes place on $S$. Pick a Poincar\'e line bundle $\mathcal P$ on $S\times\Pic$, and let $\pi\_S,\,\pi\_{\Pic}$ denote the obvious projections.

\begin{lemma} We have $\tilde\gamma=\pi\_{\Pic*}\big(\pi_S^*([\gamma])\cup
c_1(\mathcal P)\big)$, where $[\gamma]:=\mathrm{PD}(\gamma)\in H^3(S)$ is the Poincar\'e dual of $\gamma$.
\end{lemma}

\begin{proof}
Consider the K\"unneth component of $c_1(\mathcal P)\in H^2(S\times\Pic)$ in
\beq{H1}
H^1(S)\otimes H^1(\Pic)\ \cong\ \Hom(H_1(S),H^1(\Pic)).
\eeq
Identifying the right hand side of \eqref{H1} with
$$
\Hom(H_1(S),H_1(S)),
$$
its class is the identity. Thus, considered as an element of
$\Hom(H_1(S),H^1(\Pic))$ it takes $\gamma\in H_1(S)$ to $\tilde\gamma\in H^1(\Pic)$, since this is what maps back to $\gamma$ under $H^1(\Pic)\cong H_1(S)$.

However, via the isomorphism \eqref{H1}, $c_1(\mathcal P)$ takes $\gamma$ to
$$
\pi\_{\Pic*}(c_1(\mathcal P)|_{\gamma\times\Pic})
\ =\ \pi\_{\Pic*}(c_1(\mathcal P)\cup\mathrm{PD}(\gamma\times\Pic)),
$$
which is $\pi\_{\Pic*}(c_1(\mathcal P)\cup\pi_{S\,}^*[\gamma])$ as required.
\end{proof}

The following is a result of Bryan and Leung \cite[Theorem 2.1]{BL2}.

\begin{proposition} \label{BL}
The cohomology class $\pi_*f^{*\,}[\gamma]$ is $\det^*\tilde\gamma$.
\end{proposition}

\begin{proof}
We work with the diagram
\beq{penta}
\xymatrix{
& \mathcal C \ar[rd]^f\ar[ld]_\pi\ar[d]|(.45){\ \pi\times\hspace{-1pt}f} \\
\Mb \ar[d]_\det & S\times\Mb \ar[d]_\det\ar[l]^{p\_{\Mb}}\ar[r]_(.55){p\_S} & S \\
\Pic & S\times\Pic \ar[ru]_{\pi\_S}\ar[l]_(.57){\pi\_{\Pic}},}
\eeq
where $\mathcal C$ is the universal curve. On $S\times\Mb$, the divisor class $\div\mathcal C$ defines a line bundle $\O_{S\times\Mb}(\div\mathcal C)$ which, on restriction to any $S$-fibre, is isomorphic to the restriction of the pullback $\det^*\mathcal P$ of the Poincar\'e line bundle. The two therefore differ only by a line bundle pulled back from $\Mb$ (i.e. with no $S$-component). Thus for degree reasons we have
\beqa
{\det}^*\tilde\gamma &=& {\det}^*\pi\_{\Pic*}\big(\pi_S^*[\gamma]\cup
c_1(\mathcal P)\big) \\
&=& p\_{\Mb*}{\det}^*\big(\pi_S^*[\gamma]\cup c_1(\mathcal P)\big) \\
&=& p\_{\Mb*}\big(p_S^*[\gamma]\cup[\div\mathcal C]\big) \\
&=& p\_{\Mb*}\big((p_S^*[\gamma])|\_{\div\mathcal C}\big) \\
&=& \pi_*(\pi\times f)^*(p_S^*[\gamma]) \\
&=& \pi_*f^*[\gamma],
\eeqa
where in the penultimate line we have used the fact that the fundamental class of $\div\mathcal C$ is the same as that of $\mathcal C$ pushed forward by $\pi\times f$.
\end{proof}

Following Bryan and Leung \cite{BL2}, let $\gamma_i,\ i=1,\ldots,b_1=b_1(S)$ be an integral basis of $H_1(S,\Z)/$torsion, oriented so that
\beq{gamm}
\int_{\Pic}\tilde\gamma_1\wedge\ldots\wedge\tilde\gamma_{b_1}=1.
\eeq
Thus $\tilde\gamma_1\wedge\ldots\wedge\tilde\gamma_{b_1}=[pt]$ is Poincar\'e dual to a point of Pic, so by Proposition \ref{BL} and $b_1$ applications of \eqref{key},
$$
{\det}^*([pt])=r_*(\ev_1^*[\gamma_1]\wedge\ldots\wedge\ev_{b_1}^*[\gamma_{b_1}]),
$$
where $r\colon\Mb_{\!g,n+b_1}\to\Mb_{\!g,n}$ is the map that forgets the last $b_1$ points and stabilises.

Since $r$ intertwines $\ev_1,\ldots,\ev_n$ on $\Mb_{\!g,n+b_1}(S,\beta)$ and $\Mb_{\!g,n}(S,\beta)$, when we apply this to the reduced Gromov-Witten invariants \eqref{Rgbeta} we get
\begin{align} \nonumber
R_{g,\beta}(S,\sigma_1&\ldots\sigma_n[\gamma_1]\ldots[\gamma_{b_1}]) \\
&=\int_{[\Mb_{\!g,n}(S,\beta)]^{red}}\big(\ev_1^*\sigma_1\wedge\ldots\wedge
\ev_n^*\sigma_n\big)\wedge{\det}^*([pt]) \nonumber \\
&=\int_{j^![\Mb_{\!g,n}(S,\beta)]^{red}}\ev_1^*\sigma_1\wedge\ldots\wedge
\ev_n^*\sigma_n. \label{BrLe}
\end{align}
Here $j^!$ is the refined Gysin map \cite[Section 6.2]{Ful} for the inclusion of a point $j\colon\{L\}\into\Pic_\beta(S)$. \medskip

It would be nice to write $j^![\Mb_{\!g,n}(S,\beta)]^{red}$ in the form $[\Mb_{\!g,n}(S,|L|)]^{red}$. That is, we would like to see it as a reduced virtual cycle on the moduli space of stable maps whose associated divisor lies in the linear system $|L|$, defined by the Cartesian diagram
\beq{MbLdef}
\xymatrix@=18pt{
\Mb_{\!g,n}(S,|L|)\, \Into^j\ar[d] & \Mb_{\!g,n}(S,\beta) \ar[d]^{\det} \\
\{L\}\, \Into^j & \Pic_\beta(S).}
\eeq
Now we have the diagram
\beq{comm}
\xymatrix@=18pt{
N^*_{\{L\}\subset\Pic} \ar@{.>}[r]\ar[d] & j^*E\udot_{red} \ar[d] \\
\LL_{\Mb_{\!|L|}/\Mb_{\!\beta}}[-1] \ar[r] & j^*\LL_{\Mb_{\!\beta}} \ar[r] & \LL_{\Mb_{\!|L|}},}
\eeq
with $N^*_{\{L\}\subset\Pic}\cong\Omega_{\Pic}|_{\{L\}}\cong H^1(K_S)$ the conormal bundle to the point $L$ in $\Pic$. If we can fill in the dotted arrow, its cone is easily seen to give the required reduced perfect obstruction theory for $\Mb_{\!|L|}=\Mb_{\!g,n}(S,|L|)$.

The arrow is produced by simply repeating Section \ref{GWsemireg} for $H^1(K_S)$ in place of $H^0(K_S)$, giving
$$
H^1(K_S)\otimes\O_{\Mb_{\!|L|}}\To j^*E\udot_{red},
$$
analogously to the map $H^0(K_S)\otimes\O_{\Mb_{\!\beta}}\To E\udot[-1]$ of \eqref{reg}. However checking that the resulting diagram \eqref{comm} commutes appears technically difficult. We hope to return to this issue in the future.

\subsection{Point insertions and linear subsystems}
Having cut down to stable maps with image in a single linear system using insertions \eqref{BrLe}, we next show how point insertions correspond (at the level of virtual cycles) to cutting down to linear subsystems.

We work with the commutative diagram
\beq{insdg}
\xymatrix{
\mathcal C \ar[d]^a\ar@/^1pc/[rrd]^f\ar@/_2pc/[dd]_{\pi} \\
\div^*\mathcal C' \ar[r]_(.6){\div}\ar[d]^{p_2}\ar@/^1pc/[rr]^{f_2}
& \mathcal C' \ar[d]^{p_1}\ar[r]_{f_1}
& S \\
\Mb_{\!|L|} \ar[r]_{\div} & |L|,\!}
\eeq
where $\mathcal C'\to|L|$ is the universal curve over the linear system $|L|$ and $\pi\colon\mathcal C\to\Mb_{\!|L|}$ is the universal curve over the space of stable maps. Over $\Mb_{\!|L|}$ the latter maps to the former, contracting some components and replacing multiple covers by scheme-theoretic multiplicities, carrying the fundamental class of one to the other. In particular, $p_2^*=a_*\pi^*$ on homology.

Concentrating on the bottom right hand corner, we show first that
\beq{degree1}
p_{1*}f_1^*([pt])=h,
\eeq
where $[pt]\in H^4(S)$ is the Poincar\'e dual of a point of $S$, and $h\in H^2(|L|)$ is the hyperplane class. This follows from the computation
\begin{align*}
\int_{|L|}p_{1*}f_1^*([pt])h^{\dim|L|-1}&=
\int_{\PP^1}p_{1*}f_1^*([pt])=
\int_{\mathcal C'_{\PP^1}}f_1^*([pt]) \\&=
\deg\left(f_1|_{\mathcal C'_{\PP^1}}\colon\mathcal C'_{\PP^1}\to S\right)=1.
\end{align*}
Here $\PP^1\subset|L|$ is any pencil, with universal curve $\mathcal C'_{\PP^1}$ over it. Since any pencil of curves sweeps out $S$, the map $\mathcal C'_{\PP^1}\to S$ is birational and thus has degree 1, as claimed.

From \eqref{insdg} and \eqref{degree1} it follows that
$$
p_{2*}f_2^*([pt])=p_{2*}{\div}^*f_1^*([pt])={\div}^*p_{1*}f_1^*([pt])=
{\div}^*(h).
$$
Now $p_2^*=a_*\pi^*$ on homology implies that $p_{2*}=\pi_*a^*$ on cohomology. Therefore
\beq{formS}
{\div}^*(h)=\pi_*a^*f_2^*([pt])=\pi_*f^*([pt]).
\eeq
Thus we get the point insertion $\pi_*f^*([pt])$. Repeating $m$ times gives the $m$-point insertions as in \eqref{key}. In particular we get from \eqref{BrLe} that
\beq{linsys}
R_{g,\beta}(S,\sigma_1\ldots\sigma_n[\gamma\_1]\ldots[\gamma\_{b_1}][pt]^m)=
\int_{i^!j^![\Mb_{\!g,n}(S,\beta)]^{red}}\ev_1^*\sigma_1\wedge\ldots\wedge
\ev_n^*\sigma_n.
\eeq
As before $j^!$ is the generalised Gysin map for the inclusion
$j\colon\{L\}\into\Pic_\beta$, and we let $i$ be the inclusion of a codimension $m$ linear subsystem of $|L|$, where $m$ is the number of point insertions. So the above is an integral over the space of stable maps mapping to this linear subsystem.
 
\subsection{Extension to pairs and the threefold $X$}

We have concentrated on the moduli space of stable maps $\Mb_{\!g,n}(S,\beta)$, but the above results about insertions apply equally to $P_n(S,\beta)$. The proofs are the same (slightly easier even, since the universal curve, which in the stable pairs case is $ch_2(\mathbb F)$, embeds into $S\times P$ in the diagram analogous to \eqref{penta}, and the map $a$ is an isomorphism in \eqref{insdg}).
The upshot is the following analogue of \eqref{linsys}:
\beq{linsys2}
P^{red}_{n,\beta}(S,\sigma_1,\ldots\sigma_m[\gamma\_1]\ldots[\gamma_{b_1}]
[pt]^k)\ =\ 
\int_{i^!j^![P_n(S,\beta)]^{red}\ }\prod_{i=1}^m\tau(\sigma_i),
\eeq
with $i$ and $j$ as before. The reduced obstruction theory of Appendix \ref{Dmitri} is easily done relative to $\Pic_\beta(S)$ so the analogue of the commutative diagram \eqref{comm} is automatic. Therefore (though we will not need or use it) the right hand side of \eqref{linsys2} can equally be written as an integral over a reduced virtual cycle for $P_n(S,\PP):=P_n(S,|L|)
\times_{|L|}\PP$, where $i\colon\PP\into|L|$ is the linear subsystem. \medskip
 
With the obvious modifications the results also apply to $X=K_S$, and to the corresponding residue invariants. By pushing down curves in $X$ to $S$ before applying div \eqref{maps} we get maps from $P_n(X,\iota_*\beta)$ and $\Mb_{\!g,n}(X,\iota_*\beta)$ to $\Hilb_\beta(S)$ and $\Pic_\beta(S)$. Proposition \ref{BL} then holds with $[\gamma]$ replaced by its pullback to $X$, and the same therefore applies to formula \eqref{BrLe}.

Similarly formula \eqref{formS} also holds on $X$ when $[pt]$ is replaced by its pullback to $X$, i.e. by the Poincar\'e dual of a fibre of $K_S$. The same then applies to formula \eqref{linsys}.
 
\section{Counting nodal curves} \label{nodal}

\subsection{Severi degrees as reduced Gromov-Witten invariants}
In this section we show that the Severi degrees (as studied by G\"ottche \cite{Got}, for instance)
counting nodal curves in very ample linear systems can be seen as a special case of the reduced Gromov-Witten invariants \eqref{Rgbeta}. In particular we give these classical invariants
a more modern treatment using virtual cycles, allowing us to extend them to virtual counts outside of the very ample regime. 

Fix a line bundle $L$ with $H^1(L)=0$ (which almost certainly follows from the ampleness assumptions below) and $c_1(L)=\beta$ satisfying Condition \eqref{condition2}; in particular then $H^2(L)=0$ also.
Given a curve $C$, we let $g(C)$ denote its arithmetic genus, defined by $1-g(C):=\chi(\O_C)$ even when it is not reduced or connected. When $C$ is reduced its geometric genus $\overline g(C)$ is defined to be $g(\overline C)$, the genus of its normalisation.
Finally let $h$ denote the arithmetic genus of curves in $|L|$, so that
$2h-2=\beta^2-c_1(S).\beta$. 
 
\begin{proposition} \label{vample}
If $L$ is a $(2\delta+1)$-very ample line bundle on $S$ then the general $\delta$-dimensional linear system $\PP^\delta\subset|L|$ contains a finite number of \emph{irreducible} $\delta$-nodal curves appearing with multiplicity 1, and all other curves are reduced and irreducible with geometric genus $\overline g>h-\delta$. \end{proposition}

\begin{proof}
This result without the irreducibility requirement is proved in \cite[Proposition 2.1]{KST} under the weaker assumption of $\delta$-very ampleness.

So to finish we assume for a contradiction that there exists a \emph{reducible} curve in $\PP^\delta\subset |L|$. Since it must be reduced, we can write it as $A+B$, with $A$ and $B$ nonzero and having no common irreducible components.

By the Hodge index theorem, $A^2\le(L.A)^2/L^2$ for any positive $L\in H^{1,1}(S)$ and arbitrary $A\in H^{1,1}(S)$. (Proof: $A-(L.A)L/L^2$ is orthogonal to $L$ so has square $\le0$.) Applied to our situation we get
\beq{hit}
A.B=A.(L-A)\ge A.L-\frac{(A.L)^2}{L^2}=\frac{(A.L)(B.L)}{L^2}\,.
\eeq
By symmetry we may assume that $A.L\le B.L$. Then 
$$
L^2=A.L+B.L\le 2B.L
$$
so that by \eqref{hit},
\beq{A.B}
A.B\ge\frac{A.L}2\ge\frac{2\delta+1}2\,.
\eeq
The inequality $A.L \geq 2 \delta+1$ follows from the $(2\delta+1)$-very ampleness of $L$ as follows. It suffices to show the inequality for any irreducible effective divisor $A$. Choose $2\delta+2$ smooth points on $A$. By the definition of $(2\delta+1)$-very ampleness, there is a divisor in $|L|$ which passes through the first $2\delta+1$ points, but not the last one. Therefore the divisor does not contain $A$, and $L.A\ge2\delta+1$, as required.
 

But \eqref{A.B} implies the normalisation of $A+B$ pulls apart $>\delta$ intersection points, which makes the geometric genus of $A+B$ less than $h-\delta$, a contradiction.
\end{proof}

We call these irreducible $\delta$-nodal curves
\beq{Gtt}
D_i,\quad i=1,\ldots,n_\delta(L).
\eeq
Here $n_\delta(L)$ is the intersection of $\PP^\delta$ with the Severi variety
$$
\overline{\{C\in|L|\colon C \mathrm{\ has\ }\delta
\mathrm{\ nodes}\}}\ \subset\ |L|,
$$
i.e. it is the Severi degree \eqref{Sevdeg} studied by G\"ottsche \cite{Got}.

Since they are irreducible, the normalisation maps $\Db_i\to D_i$ define stable maps $f_i\colon\Db_i\to S$ from smooth connected curves. In fact these are all of the points of $\Mb_{\!h-\delta}(S,\PP^\delta):=\div^{-1}(\PP^\delta)$:

\begin{proposition}
Let $L$ be a $(2\delta+1)$-very ample line bundle on $S$. The only points of $\Mb_{\!h-\delta}(S,\beta)$ whose divisor class lies in $\PP^\delta$ are the normalisations of the $\delta$-nodal curves $D_i$. These are smooth points of $\Mb_{\!h-\delta}(S,\beta)$, and smooth isolated points of $\Mb_{\!h-\delta}(S,\PP^\delta)$.
\end{proposition}

\begin{proof} Choose a stable map $f\colon C\to S$ whose divisor class lies in $\PP^\delta$. So $C$ is connected and at worst nodal. By Proposition \ref{vample} its image $\Sigma:=f_*C$ is reduced and irreducible of geometric genus $\ge h-\delta$, and $f$ is generically one-to-one except on contracted irreducible components.

Let $C_1,\ldots C_k$ denote the irreducible components of $C$ which are contracted, and let $C_{k+1}$ denote the one which surjects onto $\Sigma$. Therefore $\overline\Sigma=\overline C_{k+1}$ and
$\overline C=\overline C_1\sqcup\ldots\sqcup\overline C_{k+1}$.

We want to show that $g(C)\ge h-\delta$ with equality only if $(C,f)$ is the normalisation of one of the $\delta$-nodal curves $D_i$. We have
$$
g(C)-1=\sum_{i=1}^{k+1}(g(\overline C_i)-1)+d,
$$
where $d=g(C)-g(\overline C)$ is the number of nodes of $C$, and
$$
g(\overline\Sigma)-1=g(\overline C_{k+1})-1.
$$
Thus
\beq{gg}
g(C)-g(\overline\Sigma)=\sum_{i=1}^k(g(\overline C_i)-1)+d.
\eeq
Contracted components of genus $g(\overline C_i)\ge2$ contribute strictly positively to \eqref{gg}. Contracted components of genus 1 must contain one of the nodes of $C$ by stability (or connectedness), forcing $d>0$ and therefore also contributing strictly positively to \eqref{gg}. Finally contracted $\PP^1$s must contain at least 3 of the nodes in $C$, by stability. So if there are $p>0$ contracted $\PP^1$s then there must be $\ge3p$ preimages of nodes upstairs on $\overline C$ and so $\ge3p/2$ nodes downstairs on $C$. In particular $d\ge3p/2>p$, so again \eqref{gg} is strictly positive.

It follows that
$$
g(C)\ge g(\overline\Sigma)\ge h-\delta,
$$
with equality implying that there are no contracted components, no nodes ($d=0$), and, by Proposition \ref{vample}, $\Sigma$ must be one of the $\delta$-nodal curves $D_i$. Thus $C$ is smooth and $f$ is the normalisation of the image $D_i$. \medskip

Finally we deal with the deformation theory of a nodal curve $D\subset S$. The ``multiplicity 1" statement of Proposition \ref{vample} refers to the scheme structure on the locus of nodal curves defined by locally pulling back the reduced scheme structure from the miniversal deformation space of the singularity \cite[proof of Proposition 2.1]{KST}. What is in fact proved is that the locus of $\delta$-nodal curves is smooth of codimension $\delta$, and then $\PP^\delta$ is chosen transverse to it. Equivalently, the composition
\beq{defcomp}
T_D\PP^\delta\to H^0(\O_D(D))\to H^0(\O_Z(D))
\eeq
is surjective. Here $Z\subset D$ is the singular set of $D$ -- the union of its nodes -- and $H^0(\O_Z(D))$ is its miniversal deformation space.

We want to relate this to the deformation theory of the resulting stable map $f\colon\Db\to S$ given by normalising $D$. Since $f$ is an immersion, $T_{\Db}\to f^*T_S$ is an injection; we define $N_f$ to be its cokernel. By local calculation,
$$
f_*N_f\cong\I_Z(D)\subset\O_D(D),
$$
where $Z\subset D$ is the union of the nodes of $D$. The usual stable map deformation-obstruction theory \eqref{Edot} reduces in this case to
$$
(E\udot)^\vee=R\Gamma\big(T_{\Db}\to f^*T_S\big)=R\Gamma(\Db,N_f)=R\Gamma(D,f_*N_f)
=R\Gamma(D,\I_Z(D)).
$$
In particular the vector space of first order deformations is $H^0(\I_Z(D))$,
and those with divisor class in $\PP^\delta$ are given by the intersection in $H^0(\O_D(D))$ of $T_D\PP^\delta$ and $H^0(\I_Z(D))$.
That is, the Zariski tangent space to $\Mb_{\!h-\delta}(S,\PP^\delta)$ at $(D,f)$ is the kernel of the composition \eqref{defcomp}. But that map is surjective between vector spaces of dimension $\delta$, so has kernel 0. Therefore the $(\Db_i,f_i)\in\Mb_{\!h-\delta}(S,\PP^\delta)$ are isolated points as claimed.

From this also follows the fact that $\Mb_{\!h-\delta}(S,\beta)$ is smooth at $(\Db_i,f_i)$. In fact let $d$ denote the dimension of its Zariski tangent space at $(\Db_i,f_i)$, and let $v$ be its virtual dimension \eqref{redvd} in the reduced obstruction theory. As for any space with perfect obstruction theory, we have
$$d\ge v,$$
with equality only if $\Mb_{\!h-\delta}(S,\beta)$ is smooth at $(\Db_i,f_i)$.
Cutting down by $h^{0,1}(S)$ equations to $\Mb_{\!h-\delta} (S,|L|)$ and then by a further $\chi(L)-1-\delta$ equations to $\Mb_{\!h-\delta}(S,\PP^\delta)$, the Zariski tangent space has complex dimension at least
$$
v-h^{0,1}(S)-\chi(L)+1+\delta=0.
$$
But the $(\Db_i,f_i)$ are isolated in $\Mb_{\!h-\delta}(S,\PP^\delta)$, so the above inequalities are both equalities, and the $(\Db_i,f_i)$ indeed define smooth points of $\Mb_{\!h-\delta}(S,\beta)$.

(Alternatively one can compute by the deformation theory above that the \emph{reduced} obstruction space vanishes.)
\end{proof}

Combining this with \eqref{linsys} we find that the reduced Gromov-Witten invariants with insertions
\beq{insertS}
(\ldots):=\big(S,[\gamma\_1]\ldots[\gamma\_{b_1(S)}][pt]^{\chi(L)-1-\delta}\big)
\eeq
equal
$$
R_{h-\delta,\beta}(\ldots)=n_\delta(L).
$$
G\"ottsche conjectured that the numbers $n_\delta(L)$ should be degree-$\delta$ polynomials in the four numbers $L^2,K_S.L,K_S^2,c_2(S)$ when $L$ is at least $(5\delta-1)$-very ample. This was proved by Tzeng \cite{Tze}, and later in \cite{KST} for $L$ at least $\delta$-very ample using the stable pair methods of the next Section.

By Lemma \ref{X=S} the reduced residue Gromov-Witten invariants of $X$ contain the same numbers. (When working on $X$ we use the insertions
\beq{insertX}
(\ldots):=\big(X,q^*[\gamma\_1]\ldots q^*[\gamma\_{b_1(S)}]q^*[pt]^{\chi(L)-1-\delta}\big)
\eeq
pulled back from \eqref{insertS} by the projection $q\colon X\to S$.)
In fact for degree reasons the other terms in Lemma \ref{X=S} all vanish. Summarising then, we have:

\begin{theorem} \label{Gottcha}
Let $L$ be a $(2\delta+1)$-very ample line bundle with $H^1(L)=0$ and $c_1(L)=\beta$ satisfying Condition \eqref{condition2}. Then the reduced Gromov-Witten invariants of both $(S,\beta)$ and $(X=K_S,\iota_*\beta)$ include the 
Severi degrees $n_\delta(L)$ of $\delta$-dimensional linear systems in $|L|$:
\beqa
R_{h-\delta,\beta}(\ldots) &=& n_\delta(L), \\
\curly R_{h-\delta,\beta}(\ldots) &=& n_\delta(L).\,t^{h-\delta-1+\int_\beta c_1(S)}.
\eeqa
Here $(\ldots)$ denotes either of the insertions \eqref{insertS} on $S$ or \eqref{insertX} on $X$, and $n_\delta(L)$ \eqref{Gtt} is a universal degree-$\delta$ polynomial in $\beta^2,\,\int_\beta c_1(S),\,c_1(S)^2$ and $c_2(S).\hfill\square$
\end{theorem}

Therefore $R_{h-\delta,\beta}(S,[\gamma\_1]\ldots[\gamma\_{b_1(S)}][pt]^{\chi(L)-1-\delta})$ gives us one way of extending 
Severi degrees to the case when $L$ is not very ample, and has possibly nonvanishing $H^1$. It is a virtual count of irreducible $\delta$-nodal curves satisfying the incidence conditions.

\subsection{Severi degrees as reduced stable pair invariants}

One can encode the reduced residue Gromov-Witten invariants $\curly R_{g,\beta}$ in BPS form. By Proposition \ref{vample} all curves in $\PP^\delta$ are reduced and irreducible. In particular, the stable maps in $\Mb_{\!g}(S,\PP^\delta)$ all have irreducible image and involve no multiple covers. So the universal formula for the BPS invariants $r_{g,\beta}\in\Q(t)$ reduces to
\beq{redGV}
\sum_{g=0}^\infty\curly R_{g,\beta}(\ldots)u^{2g-2}=\sum_{g=0}^\infty r_{g,\beta}
(\ldots)(2\sin u/2)^{2g-2},
\eeq
where we use the same insertions $(\ldots)$ as above (\ref{insertS}, \ref{insertX}).
Via the Gopakumar-Vafa, MNOP and stable pairs conjectures \cite{GV, MNOP, PT1}, all extended to the reduced and equivariant cases, the $r_{g,\beta}$ defined by \eqref{redGV} should lie in $\Z(t)$ and can also be calculated via universal formulae in the reduced residue stable pair invariants.

The leading $u^{2h-2\delta-2}$ term in \eqref{redGV} states that in genus $h-\delta$, the reduced BPS invariants are just the reduced residue Gromov-Witten invariants, which by Lemma \ref{X=S} are the reduced Gromov-Witten invariants up to a shift in the equivariant parameter:
\begin{align} \nonumber
r_{h-\delta,\beta}(\ldots)=\curly R_{h-\delta,\beta}(\ldots)
&=R_{h-\delta,\beta}(\ldots).\,t^{\!h-\delta-1+\int_\beta c_1(S)} \\
&=n_\delta(L).\,t^{h-\delta-1+\int_\beta c_1(S)}. \label{BPSGott}
\end{align}
\medskip

The MNOP and Gopakumar-Vafa conjectures state that this should equal a linear combination of stable pair invariants by the universal formulae of \cite{PT1,PT3}. Again things simplify in our case to
\beq{danny}
\sum_{i=0}^\infty r_{i,\beta}(\ldots)q^{1-i}(1+q)^{2i-2}\ =\
\sum_{i=1-h}^\infty\curly P^{red}_{\!i,\beta}(\ldots)q^i.
\eeq
These equations can be inverted to define the $r_{i,\beta}$ as linear combinations of the $\curly P^{red}_{\!i,\beta}$:
\beq{bps-1}
r_{i,\beta}(\ldots)=\left\{\!\!\begin{array}{ll}
0 & i>h, \\
\curly P^{red}_{\!1-h,\beta}(\ldots) & i=h, \\
\curly P^{red}_{\!1-i,\beta}(\ldots)-\sum_{k=i+1}^h\left(\!\!
\begin{array}{c} 2k-2 \\ k-i\end{array}\!\!\right)r_{k,\beta}(\ldots)\quad & i<h.
\end{array}\right.
\eeq
So one way of stating the conjecture is that the $r_{i,\beta}$ defined by (\ref{bps-1}) agree with the $r_{i,\beta}$ defined by \eqref{redGV}.

In this section we will prove this conjecture for the BPS number
$r_{h-\delta,\beta}$, thus showing that the 
Severi degrees are also given by a linear combination of reduced stable pair invariants.
This was the motivation behind the paper \cite{KST}.

We recall from \cite[Proposition B.8]{PT3} that $P_i(S,\beta)$ is the relative Hilbert scheme of $i-1+h$ points on the fibres of the universal curve over $\Hilb_\beta(S)$. (We describe this isomorphism in \eqref{Rhilb} in Appendix \ref{Dmitri}.) Similarly
$$
P_i(S,\PP^j)=\Hilb^{i-1+h}(\mathcal C/\PP^j)
$$
for any linear system $\PP^j\subset|L|$. And these spaces are smooth for general $\PP^j$ \cite[Section 4]{KST}.

\begin{theorem} \label{GottchaP}
Let $L$ be a $(2\delta+1)$-very ample line bundle with $H^1(L)=0$ and $c_1(L)=\beta$ satisfying Condition \eqref{condition} \emph{(\!}respectively Condition \eqref{condition2}\emph{)}.

The reduced stable pair invariants $\curly P^{red}_{\!n,\beta}(S)$ \emph{(\!}respectively
$\curly P^{red}_{\!n,\,\iota_*\beta}(X))$ contain the 
Severi degrees $n_\delta(L)$ of $\delta$-dimensional linear systems in $|L|$. That is, if we define $r_{i,\beta}(\ldots)\in\Z(t)$ by
\beq{bpdef}
\sum_{i=0}^\infty r_{i,\beta\,}(\ldots)q^{1-i}(1+q)^{2i-2}=
\sum_{i=1-h}^\infty\curly P^{red}_{\!i,\beta}(\ldots)q^i,
\eeq
with $(\ldots)$ either of the insertions (\ref{insertS}, \ref{insertX}), then
$$
r_{h-\delta,\beta}(\ldots)=n_\delta(L).\,t^{h-\delta-1+\int_\beta c_1(S)},
$$
given by a universal degree-$\delta$ polynomial in $\beta^2\!,\,\int_\beta c_1(S),\,c_1(S)^2$ and $c_2(S)$.
\end{theorem}

\begin{remark}
We use Appendix \ref{Dmitri} to deal with curve classes satisfying only Condition \eqref{condition}. If we work only with classes satisfying \eqref{condition2} then the reduced Gromov-Witten invariants are also defined and Theorem \ref{Gottcha} then shows that $R_{h-\delta,\beta}(\ldots)$ is equal to the same linear combination of reduced stable pair invariants. Thus the reduced MNOP conjecture is true in this special case.
\end{remark}

\begin{remark}
We can use \eqref{bpdef} to define virtual Severi degrees $n_\delta(L)$ outside of the very ample case. In the sequel \cite{KT2}
we show that these virtual numbers are governed by the G\"ottsche polynomials \cite{Got} just as in the very ample case \cite{Tze,KST}. In fact we prove that reduced stable pair invariants of surfaces can be calculated in terms of topological numbers much more generally.
\end{remark}

\begin{proof}
The choice of insertions together with Proposition \ref{vample} ensure the relevant stable pair moduli spaces for $X$ and $S$ are the same. We work in this proof with $S$.

Let $c=\chi(L)-1-\delta$ denote the codimension of $\PP^\delta\subset|L|$, and let $i$ run between $1-h$ and $1-h+\delta$.

Since $P_i(S,\PP^\delta)\subset P_i:=P_i(S,\beta)$ is smooth of the right reduced virtual dimension \eqref{Predvd}
$v-h^{0,1}(S)-c=i-(1-h)+\delta$, the reduced obstruction space of Proposition \ref{PTloc} (restricted to $P_i(S,\PP^\delta)$) vanishes. Thus the ordinary obstruction sheaf (restricted to $P_i(S,\PP^\delta)$) is the constant bundle with fibre $H^2(\O_S)$, and the fixed obstruction theory \eqref{pafix} of $P_i$ (restricted to $P_i(S,\PP^\delta)$) has
$$
h^0((E_X\udot|_{P_S})^{fix})=\Omega_{P_i}, \qquad h^{-1}((E_X\udot|_{P_S})^{fix})
=H^0(K_S).
$$
Therefore by \eqref{pamov} the virtual normal bundle (restricted to $P_i(S,\PP^\delta)$) similarly has
$$
h^1(N^{vir})=\Omega_{P_i}\otimes\mathfrak t,
\qquad h^0(N^{vir})=H^0(K_S)\otimes\mathfrak t,
$$
where $\mathfrak t$ is the one-dimensional representation of $T$ of weight $1$.

Substituting into the virtual localisation formula \eqref{Pres} and using the insertion formula \eqref{linsys2} (and smoothness) we find
\begin{align} \nonumber
\curly P^{red}_{\!i,\beta}\big(S,[\gamma\_1]\ldots[\gamma\_{b_1(S)}][pt]^c\big)
&=\int_{P_i(S,\PP^\delta)}\frac{c_{top}(\Omega_{P_i}\otimes\mathfrak t)}{c_{top}
(H^0(K_S)\otimes\mathfrak t)} \\ \label{brack}
&=\left(\int_{P_i(S,\PP^\delta)}c_{\dim P_i(S,\PP^\delta)}(\Omega_{P_i})\right)
t^\alpha.
\end{align}
Here $\alpha=(\rk\Omega_{P_i}-\dim P_i(S,\PP^\delta))-h^{2,0}(S)$, the first term coming from the numerator, the second from the denominator. This is
$$
\alpha=c+h^{0,1}(S)-h^{2,0}(S)=\chi(L)-\chi(\O_S)-\delta=h-1+\int_\beta c_1(S)-\delta,
$$
as required. So we concentrate on the bracketed integral in \eqref{brack}.
Consider the exact sequence $0\to N^*\to\Omega_{P_i}\to
\Omega_{P_i(S,\PP^\delta)}\to0$ on $P_i(S,\PP^\delta)$, where $N^*$ is the conormal bundle of $\PP^\delta\subset\mathrm{Hilb}_\beta(S)$, which in turn sits inside an exact sequence $0\to\Omega_{\mathrm{Pic}_\beta(S)}|_{\{L\}}\to N^*\to\O(-H)^{\oplus c}\to0$. We obtain
\begin{align*}
\int_{P_i(S,\PP^\delta)}&c_{\bull}(\Omega_{P_i(S,\PP^\delta)})
\left(1-cH+\frac{c(c-1)}2H^2-\ldots\right) \\
&=(-1)^{i+h-1+\delta}e(P_i(S,\PP^\delta)) + \sum_{j=1}^\delta (-1)^j
\left(\!\!\begin{array}{c} c \\ j\end{array}\!\!\right)\int_{P_i(S,\PP^{\delta-j})}
c_{\bull}(\Omega_{P_i(S,\PP^\delta)}).
\end{align*}
Here the $\PP^{\delta-j}$s are generic linear subspaces of $\PP^\delta\subset|L|$ chosen to contain only curves of geometric genus $\ge h-\delta+j$ (cf. Proposition \ref{vample}) and so that the $P_i(S,\PP^{\delta-j})$ are still smooth.

Similarly using the exact sequence $0\to \O(-H)^{\oplus j} \! \to\Omega_{P_i(S,\PP^\delta)} \! \to\Omega_{P_i(S,\PP^{\delta-j})}\to0$ on $P_i(S,\PP^{\delta-j})$, we find inductively that the whole expression can be written in terms of topological Euler characteristics
$$
(-1)^{i+h-1+\delta}\left(e(P_i(S,\PP^\delta))+\sum_{j=1}^\delta a_je(P_i(S,\PP^{\delta-j}))\right),
$$
for some integral coefficients $a_j$.

Therefore the $r_{i,\beta}\in\Z(t)$ are defined by setting $\sum_{i=0}^\infty r_{i,\beta}q^{1-i}(1+q)^{2i-2}$ equal to $t^{h-1+\int_\beta c_1(S)-\delta}$ times
\beq{last}
\sum_{i=1-h}^{1-h+\delta} (-1)^{i+h-1+\delta} \!\left(\!e(P_i(S,\PP^\delta))+\sum_{j=1}^\delta a_je(P_i(S,\PP^{\delta-j}))\!\right)\!q^{i} + O(q^{2-h+\delta}).
\eeq
Evaluating these Euler characteristics fibrewise, we express them as Euler characteristics of the base $\PP^{\delta-j}$s, weighted by the constructible function whose value at a point is the Euler characteristic of the fibre above it. Then by inverting the above formula as in \eqref{bps-1} we also get each $r_{i,\beta}$ as a sum of weighted Euler characteristics of the $\PP^{\delta-j}$s.

By \cite[Theorem 5]{PT3}, the weighting function for $r_{i,\beta}$ at a point $C\in\PP^{\delta-j}$ is zero unless $i$ lies between the arithmetic and geometric genera of the reduced irreducible curve $C$. In particular, for $i=h-\delta$, the function is identically zero on $\PP^{\delta-j}$ for $j>0$, and nonzero on $\PP^\delta$ only at the $\delta$-nodal curves $D_i$. And by \cite[Proposition 3.23]{PT3}, it takes the value 1 on the $D_i$.

So in \eqref{last} only the first terms contribute to $r_{h-\delta,\beta}$; the second terms all cancel to give zero. And the first Euler characteristic contributes the number of $\delta$-nodal curves $D_i$, yielding
$$
r_{h-\delta,\beta}(\ldots)=n_\delta(L).\,t^{h-1+\int_\beta c_1(S)-\delta}
$$
as claimed.
\end{proof}

\appendix
\addtocontents{toc}{\SkipTocEntry}
\section{\\ Reduced obstruction theory for stable pairs revisited, \\
\emph{by Martijn Kool, Dmitri Panov and Richard Thomas}}
\addtocontents{toc}{\protect\contentsline{section}%
  {\protect\tocsection{}{\thesection}%
    {Reduced obstruction theory for stable pairs revisited}}%
  {\thepage}} \label{Dmitri}

\pagestyle{myheadings}
\markleft{M. KOOL, D.PANOV AND R. P. THOMAS}

\subsection{The Hilbert scheme of curves as a zero locus}\hspace{-2mm}\footnote{We now realise this construction was discovered many years ago by D\"urr, Kabanov and Okonek \cite{DKO}. In the next Section A.2 we extend it to stable pairs.}
Let $H_\beta:=\Hilb_\beta(S)$ denote the Hilbert scheme of curves in class $\beta$. We describe an embedding of $H_\beta$ in a smooth space with a bundle over it and a canonical section whose zero locus is precisely $H_\beta$. For now we make an assumption slightly stronger than Condition \eqref{condition}:
\beq{assume}
H^2(L)=0 \quad\mathrm{for\ all\ line\ bundles\ with\ \ }c_1(L)=\beta.
\eeq
Fix a curve\footnote{The constructions of this Section can also be done in projective families of surfaces by picking a divisor $A$ on the whole family which has the required ampleness on each fibre.} $A\subset S$ which is sufficiently ample in the sense that
\beq{cond9}
H^{\ge1}(L(A))=0 \quad\mathrm{for\ all\ line\ bundles\ with\ \ }c_1(L)=\beta.
\eeq
Letting $\gamma=[A]+\beta$, we get an embedding $H_\beta\into H_\gamma$ defined on points by
\begin{eqnarray} \label{+A}
\Hilb_\beta(S) &\into& \Hilb_\gamma(S), \\
C\quad &\mapsto& \ A+C. \nonumber
\end{eqnarray}
At the level of schemes, the map is defined using the usual universal diagram
\beq{univ}
\xymatrix{
\mathcal C_{\ } \ar@{^(->}[r]^(.3)i\ar[dr]_{\pi_H} & S\times\Hilb_\beta(S)
\ar[r]^(.7){\pi_S} \ar[d]^{\pi_H} & S \\
& \Hilb_\beta(S).}
\eeq
The divisor $\mathcal C+\pi_S^*A\subset S\times H_\beta$ is a flat family of divisors in $S$ of class $\gamma$ so has a classifying map from the base $\Hilb_\beta(S)$ to $\Hilb_\gamma(S)$. \medskip

Notice that by \eqref{cond9}, $H_\gamma$ is smooth: it is a projective bundle over $\Pic_\gamma(S)$. Let $\mathcal D\subset S\times H_\gamma$ be the universal divisor, and again use $\pi_H$ to denote the projection to $H_\gamma$. Consider $\pi_{H*}$ applied to the exact sequence
$$
0\To\O(\mathcal D-\pi_S^*A)\To\O(\mathcal D)\To\O(\mathcal D)|_{\pi_S^*A}\To0.
$$
$R^2\pi_{H*}$ of the first term vanishes by the assumption \eqref{assume}, and $R^{\ge1}\pi_{H*}$ of the central term vanishes by \eqref{cond9}. Therefore $R^{\ge1}\pi_{H*}$ of the final term vanishes, and
\beq{vect}
F:=\pi_{H*}\big(\O(\mathcal D)|_{\pi_S^*A}\big)
\eeq
is a vector bundle on $H_\gamma$.

\begin{proposition} \label{zero}
Define the canonical section $\sigma$ of $F\to\Hilb_\gamma(S)$ by pushing down the restriction $s_{\mathcal D}|_{\pi_S^*A}$ of the canonical section of $\O(\mathcal D)$. Then the zero locus of $\sigma$ is the subscheme
$$
A+\Hilb_\beta(S)\subset\Hilb_\gamma(S).
$$
\end{proposition}

\begin{proof}
By its very definition, $\sigma(D)=0$ for $D\in\Hilb_\gamma(S)$ if and only if $s_D|_A=0$, if and only if $A\subset D$.  This gives the result at the level of sets. In fact scheme theoretically, it is also clear that $\sigma$ vanishes on $A+\Hilb_\beta(S)$. So it is sufficient to produce the inverse morphism $Z(\sigma)\to A+\Hilb_\beta(S)$.

Since $s_{\mathcal D}$ vanishes on the pullback of $A$ to $S\times Z(\sigma)$, we can divide by the pullback of its defining equation $s_A$ to give a divisor in $S\times Z(\sigma)$ whose classifying map $Z(\sigma)\to\Hilb_\beta(S)$ gives the required inverse.
\end{proof}

Letting $I$ denote the ideal of $H_\beta\subset H_\gamma$ we get the diagram
on $Z(\sigma)=H_\beta$:
\beq{Fpo1}
\xymatrix@R=20pt{
F^*|_{H_\beta} \ar[r]^(.45){d\sigma}\ar[d]^\sigma & \Omega_{H_\gamma}|_{H_\beta} \ar@{=}[d] \\ I/I^2 \ar[r]^(.45)d & \Omega_{H_\gamma}|_{H_\beta}.\!}
\eeq
We denote the upper row by the complex $\{F^{-1}\to F^0\}=:F_{red}\udot$ of vector bundles on $H_\beta$. The bottom row is the truncated cotangent complex of $H_\beta$, so we have
$$
F\udot_{red}\to\LL_{\,\Hilb_\beta(S)}.
$$
This is an isomorphism on $h^0=\Omega_{H_\beta}$ and a surjection on $h^{-1}$ because $F^*|_{H_\beta}\to I/I^2$ is onto. \medskip

We will see that $F\udot_{red}$ is the reduced obstruction theory. As a warm up we explain what happens at a point $C\in\Hilb_\beta(S)$. The obvious long exact sequence
\begin{align*}
0\to H^0(\O_C(C))\to & H^0(\O_{A+C}(A+C))\to H^0(\O_A(A+C))\to
\\ & H^1(\O_C(C))\to H^1(\O_{A+C}(A+C))\to H^1(\O_A(A+C))
\end{align*}
is precisely
\beq{dF}
0\to T_CH_\beta\to T_{C+A}H_\gamma\to F|_{\{C+A\}}\to H^1(\O_C(C))\to H^2(\O_S)\to0,
\eeq
using the assumptions (\ref{assume}, \ref{cond9}) and identifying the penultimate term via the exact sequence\footnote{In other words, we identify $H^1(\O_{A+C}(A+C))
\Rt{\sim}H^2(\O_S)$ by the semi-regularity map for $A+C$. But the resulting map $H^1(\O_C(C))\to H^2(\O_S)$ in \eqref{dF} is the semi-regularity map for $C$ itself, as can be seen from the comparison maps between the sequence $0\to\O_S\to\O(A+C)\to\O_{A+C}(A+C)\to0$ and the sequence of subsheaves
$0\to\O_S\to\O(C)\to\O_C(C)\to0$.} $0\to\O_S\to\O(A+C)\to\O_{A+C}(A+C)\to0$.

The first few maps in \eqref{dF} are, respectively, the derivative at $C$ of the inclusion $H_\beta\into H_\gamma$, and the derivative at $C+A$ of the section $s\in\Gamma(F)$. This latter map is exactly the complex $(F\udot_{red})^\vee|_C$ (i.e. the dual of the top line of \eqref{Fpo1}), with cokernel
$$
h^1((F\udot_{red})^\vee)=\ker\big(H^1(\O_C(C))\to H^2(\O_S)\big).
$$
Thus the obstruction space of $F\udot_{red}$ at $C$ is the \emph{reduced} obstruction space: the kernel of the semi-regularity map of \eqref{semireg} from the usual obstruction space $H^1(\O_C(C))$ to $H^2(\O_S)$.  \medskip

We now prove all this for the universal curve $\mathcal C$ over the family $\Hilb_\beta(S)$. The ordinary obstruction theory $F\udot\to\LL_{H_\beta}$, where
$(F\udot)^\vee:=R\pi_{H*}\O_{\mathcal C}(\mathcal C)$, admits the semi-regularity map
\beq{Fdot}
(F\udot)^\vee=R\pi_{H*}\O_{\mathcal C}(\mathcal C)\To R^2\pi_{H*}\O_{S\times H_\beta}[-1]
\eeq
induced by the exact sequence $0\to\O\to\O(\mathcal C)\to\O_{\mathcal C}(\mathcal C)\to0$.

\begin{proposition} \label{A2}
The semi-regularity map fits into an exact triangle intertwining the two perfect obstruction theory maps:
$$
\xymatrix@R=18pt{
(F\udot_{red})^\vee \ar[r] & (F\udot)^\vee\!\! \ar[r] & H^2(\O_S)\otimes\O_{H_\beta}[-1] \\ & \LL_{H_\beta}^\vee.\!\!\!\! \ar[u]\ar[ul]}
$$
\end{proposition}

\begin{proof} First we recall the obstruction theory of $H_\beta$. Using the diagram \eqref{univ} we get maps
\beq{ccmaps}
\O_{\mathcal C}(-\mathcal C)=\LL_{\,\mathcal C/(S\times H_\beta)}[-1]\To
Li^*\LL_{S\times H_\beta}=\pi_S^*\LL_S\oplus\pi_H^*\LL_{H_\beta}\To
\pi_H^*\LL_{H_\beta}.
\eeq
Their composition gives, by adjunction, the perfect obstruction theory\footnote{This construction, viewing curves as divisors in $S$, coincides with the obstruction theory obtained by thinking of $\Hilb_\beta(S)$ as parameterising pairs $(\O_C,1)$ of a sheaf and a section. The essential point is that the Atiyah class of $\O_{\mathcal C}$ in $\Ext^1(\O_{\mathcal C},\O_{\mathcal C}\otimes\LL_{S\times H_\beta})$ is the canonical morphism in the summand $\Hom(\O_{\mathcal C}(-\mathcal C),Li^*\LL_{S\times H_\beta})$.}
\beq{potHb}
\LL_{H_\beta}^\vee\To R\pi_{H*}\O_{\mathcal C}(\mathcal C).
\eeq
Similar working using $\O_{\mathcal D}(-\mathcal D)$ on $S\times H_\gamma$ gives (recalling that $H_\gamma$ is smooth),
\beq{potHg}
T_{H_\gamma}=\LL_{H_\gamma}^\vee\To R\pi_{H*}\O_{\mathcal D}(\mathcal D).
\eeq

On restriction to $\xymatrix@1{H_\beta\ \ar@{^(->}^{+A}[r] & H_\gamma}$, the divisor $\mathcal D$ pulls back to $\pi_S^*A+\mathcal C$, so we have the exact sequence
$$
0\To\O_{\mathcal C}(\mathcal C)\To\O_{\mathcal D}(\mathcal D)\To
\O_{\pi_S^*A}(\mathcal D)\To0 \quad\mathrm{over}\ H_\beta.
$$
Pushing down to $H_\beta$ and combining with the maps (\ref{potHb}, \ref{potHg}) gives the commutative diagram
\beq{ober}
\xymatrix@=20pt{
\LL^\vee_{H_\beta} \ar[r]\ar[d] & T_{H_\gamma}|\_{H_\beta} \ar[d] \\
R\pi_{H*}\O_{\mathcal C}(\mathcal C) \ar[r] & R\pi_{H*}\O_{\mathcal D}(\mathcal D) \ar[r] & F,}
\eeq
where the bottom row is an exact triangle and the composition $T_{H_\gamma}|\_{H_\beta}
\to F$ is $d\sigma$. Thus we get the diagram
\beq{ober2}
\xymatrix@=16pt{
T_{H_\gamma}|\_{H_\beta} \ar[r]^(.6){d\sigma}\ar[d] & F \ar@{=}[d] \\
R\pi_{H*}\O_{\mathcal D}(\mathcal D) \ar[r]\ar[d] & F \\
R^1\pi_{H*}\O_{\mathcal D}(\mathcal D)[-1],}
\eeq
with verticals exact since $T_{H_\gamma}\Rt{\sim}\pi_{H*}\O_{\mathcal D}(\mathcal D)$. The cone on the first row is $(F\udot_{red})^\vee$ \eqref{Fpo1}, the second is $(F\udot)^\vee$ by \eqref{ober}, and the third is $H^2(\O_S)\otimes\O_{H_\beta}[-1]$ by the exact sequence $0\to\O_{S\times H_\beta}\to\O_{S\times H_\beta}(\mathcal D)\to\O_{\mathcal D}(\mathcal D)\to0$. We get the exact triangle
\beq{close}
\xymatrix@=16pt{
(F\udot_{red})^\vee \ar[d] \\
(F\udot)^\vee \ar[d] & \LL^\vee_{H_\beta} \ar[l]\ar@{.>}[ul] \\
R^2\pi_{H*}\O_{S\times H_\beta}[-1],\!\!\!\!}
\eeq
where we have added in the map \eqref{potHb}.
Though the bottom arrow was constructed from the semi-regularity map for $\mathcal D$, it is the semi-regularity map $(F\udot)^\vee\to
H^2(\O_S)\otimes\O_{H_\beta}[-1]$ for $\mathcal C$ by the commutativity of the diagram of exact triangles
$$
\xymatrix@R=18pt@C=12pt{
R\pi_{H*}\O_{S\times H_\beta}(\mathcal C) \ar[r]^(.47){s_A}\ar[d] &
R\pi_{H*}\O_{S\times H_\beta}(\mathcal D) \ar[d] \\
R\pi_{H*}\O_{\mathcal C}(\mathcal C) \ar[r]^(.41){s_A}\ar[d] & R\pi_{H*}\O_{\mathcal D}(\mathcal D)|_{S\times H_\beta} \ar[d] \\
R\pi_{H*}\O_{S\times H_\beta}[1] \ar@{=}[r] & R\pi_{H*}\O_{S\times H_\beta}[1]}
$$
and the functoriality of truncation.

Finally then we need the dotted arrow in \eqref{close}. Considering the map
$\LL^\vee_{H_\beta}\to(F\udot)^\vee$ of \eqref{potHb} as a morphism to the cone on the second row of the diagram \eqref{ober2}, the diagram \eqref{ober} shows that it factors through the complex $\{T_{H_\gamma}|_{H_\beta}\to F\}$, which is $(F\udot_{red})^\vee$ as required.
\end{proof}

\subsection{The relative Hilbert scheme of points as a zero locus} \label{cut}
Recall that the moduli space of stable pairs is a relative Hilbert scheme of points over $H_\beta$:
$$
P_{1-h+n}(S,\beta)\ \cong\ \Hilb^n(\mathcal C/H_\beta).
$$
Just as we described the base $H_\beta$ as the zero locus of a section of a bundle, we can do the same for the fibres. We use the embedding
\beq{embed}
\Hilb^n(\mathcal C/H_\beta)\ \subset\ S^{[n]}\times H_\beta,
\eeq
where $S^{[n]}:=\Hilb^n(S)$ is smooth because $S$ is a surface. There is a universal subscheme
$$
\mathcal Z\ \subset\ S\times S^{[n]}\times H_\beta\Rt{\pi}S^{[n]}\times H_\beta
$$
pulled back from $S\times S^{[n]}$, and of course a universal curve $\mathcal C$ pulled back from $S\times H_\beta$ inducing a universal line bundle and section $(\O(\mathcal C),s_{\mathcal C})$. This induces a canonical section $\sigma_{\mathcal C}$ of the rank $n$ vector bundle
\beq{budl}
\O(\mathcal C)^{[n]}:=\pi_*\big(\O(\mathcal C)|_{\mathcal Z}\big).
\eeq
Its zeros are the pairs $(Z,C)$ with $Z\subset C\subset S$; in fact by the argument of Proposition \ref{zero} it is scheme theoretically $\Hilb^n(\mathcal C/H_\beta)$ embedded as in \eqref{embed}.

Thus we get a perfect relative obstruction theory in the usual way:
\beq{prot}
\xymatrix{
\hspace{-1cm}\mathcal E\udot:=\big\{(\O(\mathcal C)^{[n]})^*
\ar[r]^(.64){d\sigma_{\mathcal C}}\ar[d]^{\sigma_{\mathcal C}} &
\Omega_{S^{[n]}}\big\}\hspace{-4mm} \ar@{=}[d] \\
\hspace{-33mm}\LL_{\Hilb^n(\mathcal C/H_\beta)/H_\beta}=\big\{I/I^2 \ar[r]^d & \Omega_{S^{[n]}}\big\}.\hspace{-5mm}}
\eeq

\subsection{Identifying the obstruction theories}
Let $P$ denote $P_{1-h+n}(S,\beta)\cong\Hilb^n(\mathcal C/H_\beta)$. We now have a perfect relative obstruction theory $\mathcal E\udot$ for $P/H_\beta$ by \eqref{prot}, and ordinary and reduced perfect obstruction theories $F\udot,\,F\udot_{red}$ for $H_\beta$ (Proposition \ref{A2}) when Condition \eqref{condition} holds.

We would like to combine them to give ordinary and perfect absolute obstruction theories for $P$ under Condition \eqref{condition}, and to know that they coincide with the ordinary and reduced perfect obstruction theories $E\udot,\,E\udot_{red}$ of Proposition \ref{PTloc} when Condition \eqref{condition2} is also satisfied.

That is, we want commutative diagrams of exact triangles:
\beq{combine}
\xymatrix@=18pt{
E\udot \ar@{.>}[r]\ar[d] & \mathcal E\udot \ar@{.>}[r]\ar[d] & F\udot[1] \ar[d] &\mathrm{and}& E\udot_{red} \ar@{.>}[r]\ar[d] & \mathcal E\udot \ar@{.>}[r]\ar[d] & F\udot_{red}[1] \ar[d] \\ 
\LL_P \ar[r] & \LL_{P/H_\beta} \ar[r] & \LL_{H_\beta}[1] &&
\LL_P \ar[r] & \LL_{P/H_\beta} \ar[r] & \LL_{H_\beta}[1].}
\eeq
We will first define the upper exact triangles, then prove exactness of the lower rows, and finally turn to commutativity of the resulting squares.

For ease of exposition, in this section we will work at one point $(F,s)$ of $P$ at a time. As in the rest of the paper, all of the arguments extend over the whole family $P$ in the usual way, modulo more cumbersome notation.

\begin{proposition} \label{A3}
The obstruction theory $E\udot=(R\Hom(I\udot,F))^\vee$ of Proposition \ref{PTloc} sits in a natural exact triangle with the relative obstruction theory $\mathcal E\udot$ of $P/H_\beta$ \eqref{prot} and the obstruction theory $F\udot$ \eqref{Fdot} for $H_\beta$:
$$
F\udot\To E\udot\To\mathcal E\udot.
$$
\end{proposition}

%
%
%

To prove this we need to recall from \cite[Appendix B]{PT3} the construction of the isomorphism between the moduli space of stable pairs and the relative Hilbert scheme of points on the fibres of $\mathcal C$,
\beq{Rhilb}
P_{1-h+n}(S,\beta)\cong\Hilb^n(\mathcal C/H_\beta).
\eeq
Given a pair $(F,s)$ with scheme-theoretic support $C\in H_\beta$, we dualise
$$
0\To\O_C\To F\To Q\To0
$$
considered \emph{as sheaves and maps of sheaves on} $C$. That is, applying $(\ \cdot\ )^*:=\hom_C(\ \cdot\ ,\O_C)$ gives
$$
0\To F^*\To\O_C\To\ext^1_C(Q,\O_C)\To0,
$$
with all higher Ext sheaves zero \cite[Proposition B.5]{PT3}. Therefore $F^*$ is an ideal sheaf on $C$; denoting the corresponding subscheme by $Z\subset C$ we can write the above sequence as
$$
0\To I_{Z\subset C}\To\O_C\To\O_Z\To0.
$$
Then $Z\in\Hilb^nC$ defines our point $(Z,C)\in\Hilb^n(\mathcal C/H_\beta)$. 
\begin{lemma} \label{A4}
The derived dual of $i_*F$ is $i_*(I_{Z\subset C})(C)[-1]$. The derived dual of the complex $I\udot:=\{\O_S\Rt{s}i_*F\}$ is the ideal sheaf of $Z\subset S$ twisted by the line bundle $\O_S(C)$:
$$
(I\udot)^\vee\cong\I_Z(C).
$$
\end{lemma}

\begin{proof}
Since $C$ is a Cartier divisor in $S$, it is Gorenstein with canonical bundle $K_S(C)|_C$. Therefore Serre duality for the inclusion $i\colon C\into S$ gives
\begin{align*}
(i_*F)^\vee=R\hom(i_*F,\O_S) \cong i_*R\hom&(F,i^!\O_S)\cong
i_*R\hom(F,\O_C(C)[-1]) \\
&\cong i_*(F^*)(C)[-1]\cong i_*(I_{Z\subset C})(C)[-1].
\end{align*}
Dualising the commutative diagram of exact triangles
$$
\xymatrix@=18pt{
\O_S(-C) \ar[r]^(.6){s\_C}\ar[d] & \O_S \ar[r]\ar@{=}[d]<-.5ex> & \O_C \ar[d]^s \\
I\udot \ar[r] & \O_S \ar[r]^s & i_*F}
$$
yields
$$
\xymatrix@=18pt{
\O_S \ar[r]^(.4){s\_C}\ar@{=}[d]<-.5ex> & \O_S(C) \ar[r] & \O_C(C) \\
\O_S \ar[r] & (I\udot)^\vee \ar[r]\ar[u] & i_*I_{Z\subset C}(C). \ar[u]}
$$
The top row is the obvious exact sequence defined by the canonical extension class in $\Ext^1(\O_C(C),\O_S)$. Pulling this back via the right hand arrow $i_*I_{Z\subset C}(C)\to\O_C(C)$ gives the extension class in $\Ext^1(i_*I_{Z\subset C},\O_S)$ of the bottom row. Therefore $(I\udot)^\vee$ is the kernel of the induced map from $\O_S(C)$ to the cokernel of the right hand arrow. Since this arrow is the canonical inclusion, it has quotient $\O_Z(C)$ and $(I\udot)^\vee$ is $\I_Z(C)$ as claimed.
\end{proof}

\begin{proof}[Proof of Proposition \ref{A3}]
By Lemma \ref{A4}, the ordinary deformation-obstruc\-tion complex $(E\udot)^\vee$ at $(F,s)\in P$ can be written
\beq{dualdef}
(E\udot)^\vee:=R\Hom(I\udot,F)=R\Hom(F^\vee,(I\udot)^\vee)=R\Hom_S(i_*I_{Z\subset C},\I_Z)[1]
\eeq
at $(Z\subset C)\in P$.

Via the obvious exact sequences we get a commuting diagram of exact triangles (in which some $i_*$s are suppressed, and all $R\hom$s are taken on $S$),
\beq{140}
\xymatrix@=16pt{
R\hom(\O_C,\O_Z)[-1] \ar[r]\ar[d] & R\hom(\O_C,\I_Z) \ar[r]\ar[d] & 
R\hom(\O_C,\O_S) \ar[d] \\
R\hom(i_*I_Z,\O_Z)[-1] \ar[r]\ar[d] & R\hom(i_*I_Z,\I_Z) \ar[r]\ar[d] &
R\hom(i_*I_Z,\O_S) \ar[d] \\
R\hom(\O_Z,\O_Z) \ar[r] & R\hom(\O_Z,\I_Z)[1] \ar[r] & R\hom(\O_Z,\O_S)[1].\!}
\eeq
The top degree cohomologies of the terms in the bottom right and bottom left corners are the same via the induced homomorphism between them:
\beq{ext2}
\ext^2(\O_Z,\O_S)\Rt{\sim}\ext^2(\O_Z,\O_Z).
\eeq
(For instance, this may be seen by Serre duality and the fact that the obvious map $\hom(\O_Z,\O_Z)
\to\hom(\O_S,\O_Z)$ is an isomorphism.) The same goes for the lowest degree cohomologies of the bottom left and top left terms:
\beq{ext0}
\hom(\O_Z,\O_Z)\Rt{\sim}\hom(\O_C,\O_Z).
\eeq
Thus, in the central horizontal exact triangle, we can take the cone on the map to \eqref{ext2}, then the cone on the map from \eqref{ext0}, to yield a new exact triangle
$$
\left\{\!\!\vcenter{\xymatrix{\ext^1(\O_Z,\O_Z) \ar[d] \\
\ext^1(\O_C,\O_Z)}}\!\!\right\} \hspace{-12mm}
\xymatrix{& \ar[r] & R\hom(i_*I_Z,\I_Z)[1] \ar[r] & R\hom(\O_C,\O_S)[1].}
$$
The vertical arrow is constructed from the restriction map $\O_C\to\O_Z$. Identifying $\ext^1(\O_C,\O_Z)\cong\O_Z(C)$ using $s_C$, it becomes the derivative of $s_C$. Similarly we can identify the last term $R\hom(\O_C,\O_S)$ with $\O_C(C)[-1]$. On applying $R\Gamma$ (or, in the family case, pushing down to the moduli space $P$) we get the exact triangle
$$
\left\{\!\!\vcenter{\xymatrix{T_ZS^{[n]} \ar[d] \\
\O(C)^{[n]}}}\!\!\right\} \hspace{-12mm}
\xymatrix{& \ar[r] & R\Hom(i_*I_Z,\I_Z)[1] \ar[r] & R\Gamma(\O_C(C)),}
$$
with first term $(\mathcal E\udot)^\vee$ of \eqref{prot}. By \eqref{dualdef} we have the required exact triangle
\[
\xymatrix{(\mathcal E\udot)^\vee \ar[r] & (E\udot)^\vee \ar[r] & (F\udot)^\vee.}  
\qedhere
\]
\end{proof}

Since we have been working with \emph{truncated}, rather than full, cotangent complexes, we have to be careful about exactness.

\begin{lemma}
The canonical maps of truncated cotangent complexes
$$
\LL_{H_\beta}\To\LL_P\To\LL_{P/H_\beta}
$$
define an exact triangle.
\end{lemma}

\begin{proof}
We need to show that the following triangle of complexes
$$
\big\{I/I^2|_P\to\Omega_{H_\gamma}|_P\big\}\To
\big\{\I/\I^2\to(\Omega_{S^{[n]}}\oplus\Omega_{H_\gamma})|_P\big\}\To
\big\{\mathcal J/\mathcal J^2\to\Omega_{S^{[n]}}|_P\big\}
$$
is exact. Here the ideal of $H_\beta\subset H_\gamma$ is denoted $I$, the ideal of $P\subset S^{[n]}\times H_\gamma$ is $\I$, and the ideal of $P\subset S^{[n]}\times H_\beta$ is $\mathcal J:=\I/(\I\cap I)$. The arrows are the obvious ones, and we have suppressed some (flat) pullbacks. We therefore need to prove the exactness of
$$
0\to I/(I^2+I.\I)\to\I/\I^2\to\mathcal J/\mathcal J^2\to0
$$
on the left. That is, since $I^2+I.\I=I.\I$, we need to show that
\beq{ideals}
I\cap\I^2\subseteq I.\I.
\eeq

First work on the open set in $P\cong\Hilb^n(\mathcal C/H_\beta)\ni(C,Z)$ where the $0$-dimensional subscheme $Z$ is disjoint from $A\cap C$. (Recall that $A$ was a fixed ample divisor used to embed $H_\beta$ into $H_\gamma$ \eqref{+A}.) Then in a neighbourhood in $S^{[n]}\times H_\gamma$ of this open set, we have $\I=I+J$, where $J$ is the ideal sheaf of the subscheme
\beq{RHilbA}
\Hilb^n(\mathcal D/H_\gamma)\subset S^{[n]}\times H_\gamma.
\eeq
(In other words, the locus of pairs $(Z\subset C)$ is locally isomorphic to the locus of pairs $(Z\subset(A\cup C))$ so long as $Z$ stays away from $A\cap C$.) But the relative Hilbert scheme \eqref{RHilbA} is flat\footnote{As in \cite{AIK}, the fact that each fibre has dimension $n$ follows immediately from the fact that the \emph{punctual} Hilbert scheme of length-$r$ subschemes supported at a single point of a smooth surface has dimension $r-1$ \cite{Iar}, so for a curve in a surface it has dimension $\le(r-1)$. Then, since each fibre is cut out of the smooth $2n$-dimensional space $S^{[n]}$ by a section of a rank $n$ vector bundle, this section is regular and the associated Koszul resolution exact. Thus the Hilbert polynomial of each fibre can be calculated in terms of the Chern classes of this bundle on $S$, independently of the fibre. This implies flatness.} over $H_\gamma$
 and $I$ is pulled back from $H_\gamma$, so $I\cap J^2\subseteq I\cap J=I.J$. Therefore
$$
I\cap\I^2=I\cap(I+J)^2=I\cap(I^2+I.J+J^2)=I^2+I.J+I\cap J^2\subseteq
I^2+I.J=I.\I,
$$
giving \eqref{ideals}.

This gives exactness over an open set of $P$. By changing the divisor $A$ we can cover $P$ with open sets on which we get exactness. Since exactness can be checked locally, we are done.
\end{proof}

Combining the map $\mathcal E\udot\to F\udot[1]$ of Proposition \ref{A3} with the maps \eqref{prot} and \eqref{potHb} we get the square
\beq{squ}
\xymatrix{\mathcal E\udot \ar[r]\ar[d] & F\udot[1] \ar[d] \\
\LL_{P/H_\beta} \ar[r] & \LL_{H_\beta}[1].}
\eeq

\begin{lemma} \label{abel}
The square \eqref{squ} commutes.
\end{lemma}

\begin{proof}
We use the following general fact. Suppose that a scheme $P/B$ over a (possibly singular) base $B$ is cut out of the product $A\times B$ by a section $\sigma$ of a vector bundle $E$. If we further assume that $A$ is smooth then we get a natural induced relative perfect obstruction theory
$$
\xymatrix{
\mathcal E\udot\!\!\ar[d] & \hspace{-1cm}\simeq & E^*|_P \ar[r]^{d\sigma}\ar[d]^\sigma & \Omega_A|_P \ar@{=}[d] \\ \LL_{P/B} & \hspace{-1cm}\simeq & \mathcal J/\mathcal J^2 \ar[r]^(.4)d & \Omega_{A\times B/B}|_P,}
$$
where $\mathcal J\subset\O_{A\times B}$ is the ideal sheaf of $P$.
 
Then the composition of the maps $\mathcal E\udot\to\LL_{P/B}\to\LL_B[1]$ factors through the obvious map $\mathcal E\udot\to E^*|_P[1]$ via the map
$$
E^*|_P\To\LL_B,
$$
which is the composition
$$
E^*|_P\Rt{\sigma}\mathcal J/\mathcal J^2=\LL_{P/(A\times B)}[-1]\To\LL_{A\times B}|_P\To\LL_B.
$$
Here the last map is just projection onto the second factor of $\LL_A\oplus\LL_B$, and as usual we have suppressed some pull back maps.

We use this in the setting of Section \ref{cut}, with $A=S^{[n]},\ B=H_\beta,\ E=\O(\mathcal C)^{[n]}$ and $\sigma=\sigma_{\mathcal C}$. See \eqref{budl}. Then going anticlockwise round \eqref{squ} gives the composition of
\beq{knifehandchop}
\mathcal E\udot\To(\O(\mathcal C)^{[n]})^*[1]\Rt{\sigma_{\mathcal C}}\mathcal J/\mathcal J^2[1]
\To(\LL_{S^{[n]}}\oplus\LL_{H_\beta})|_P[1]\To\LL_{H_\beta}[1].
\eeq
The second arrow is given by the canonical section $\sigma_{\mathcal C}$ of $\O(\mathcal C)^{[n]}$ \eqref{budl}. This is pushed down from the universal 0-dimensional subscheme $\mathcal Z$ inside the universal curve $\mathcal C$ over $P$:
$$
\xymatrix{
\mathcal Z\, \ar[d]_\pi\Into & \mathcal C\, \ar[dl]_\pi\Into & P\times S\, \Into & S^{[n]}\times H_\beta\times S
\\ \,P.}
$$
Namely, $\sigma_{\mathcal C}=\pi_*(s_{\mathcal C}|_{\mathcal Z})$, where $s_{\mathcal C}$ is the section of $\O(\mathcal C)$ cutting out $\mathcal C\subset S^{[n]} \times H_\beta \times S$.

We have the commutative diagram
$$
\xymatrix@=18pt{\O_{\mathcal C}(-\mathcal C) \ar@{=}[r]\ar[d] &
\I_{\mathcal C}/\I_{\mathcal C}^2 \ar[r]\ar[d] & 
\LL_{S^{[n]}\times H_\beta\times S}|_{\mathcal C} \ar[r]\ar[d] &
\pi^*\LL_{H_\beta}|_{\mathcal C} \ar[d] \\
\O_{\mathcal Z}(-\mathcal C) \ar[r]^{s_{\mathcal C}} &
\I_{\mathcal Z}/\I_{\mathcal Z}^2 \ar[r] & 
\LL_{S^{[n]}\times H_\beta\times S}|_{\mathcal Z} \ar[r] &
\pi^*\LL_{H_\beta}|_{\mathcal Z},}
$$
where in the bottom row we have suppressed the pushforward map from $\mathcal Z$ to $\mathcal C$.
Composing horizontally and dualising, by adjunction we get the commutative diagram
\beq{doormouse}
\xymatrix@=18pt{
\LL_{H_\beta}^\vee \ar[r]\ar@{=}[d] & R\pi_*\O_{\mathcal C}(\mathcal C) \ar[d] \\ \LL_{H_\beta}^\vee \ar[r] & \pi_*\O_{\mathcal Z}(\mathcal C),}
\eeq
with bottom row induced by $\pi_*(s_{\mathcal C}|_{\mathcal Z})=\sigma_{\mathcal C}$. Dualising this bottom row gives the composition of the second, third and fourth arrows in \eqref{knifehandchop}.

Going clockwise round \eqref{squ} gives
$$
\mathcal E\udot\To\big(R\pi_{H*}\O_{\mathcal C}(\mathcal C)\big)^\vee[1]\To\LL_{H_\beta}[1],
$$
where the second arrow is the dual of \eqref{potHb}. But this is the top row of \eqref{doormouse}, by construction. The first arrow comes from Proposition \ref{A3}; the top row of \eqref{140} shows it factors through
$$
\mathcal E\udot\To(\O(\mathcal C)^{[n]})^*[1]\To\big(R\pi_*\O_{\mathcal C}(\mathcal C)\big)^\vee[1],
$$
where the second arrow is the dual of $R\pi_*\O_{\mathcal C}(\mathcal C)\to R\pi_*\O_{\mathcal Z}(\mathcal C)=\O(\mathcal C)^{[n]}$.
But this is the right hand vertical arrow in \eqref{doormouse}. By the commutativity of \eqref{doormouse}, then, the two compositions are the same.
\end{proof}

It follows that $\mathrm{Cone}\big(\mathcal E\udot[-1]\to F\udot\big)$ maps to $\LL_P$ giving a perfect obstruction theory. By Proposition \ref{A3} this perfect obstruction theory is the complex
$$
R\Hom(I\udot,F)^\vee
$$
at a point $(F,s)\in P$. We \emph{have not} proven that the map from this complex to $\LL_P$ is the same as the one given by the perfect obstruction theory of Proposition \ref{PTloc}. Though they are surely the same, we don't need this to deduce that the resulting virtual cycles are the same since they depend only on the K-theory class of the 2-term complex. What we really want, however, is the same result for the reduced cycles.

The dual semi-regularity map of \eqref{Fdot} applied to the exact triangle of Proposition \ref{A3} gives the diagram
\beq{delpy}
\xymatrix@=18pt{H^{2,0}(S)[1] \ar@{=}[r]\ar[d] & H^{2,0}(S)[1] \ar[d] \\
F\udot \ar[r] & E\udot \ar[r] & \mathcal E\udot.}
\eeq
Taking cones gives
\beq{delpy2}
\xymatrix@=18pt{F_{red}\udot \ar[r] & E_{red}\udot \ar[r] & \mathcal E\udot,}
\eeq
as desired
in \eqref{combine}. Since $F_{red}\udot$ and $\mathcal E\udot$ are perfect 2-term complexes, so is $E_{red}\udot$.

By Proposition \ref{A2} the map $F\udot\to\LL_{H_\beta}$ \eqref{potHb} descends to $F\udot_{red}\to\LL_{H_\beta}$ to give the perfect obstruction theory of \eqref{Fpo1}. Therefore by Lemma \ref{abel} the perfect relative obstruction theory $\mathcal E\udot\to\LL_{P/H_\beta}$ of \eqref{prot} fits into a commutative diagram
$$
\xymatrix@=18pt{
E\udot_{red} \ar[r]\ar@{.>}[d] & \mathcal E\udot \ar[r]\ar[d] & F\udot_{red}[1] \ar[d] \\ \LL_P \ar[r] & \LL_{P/H_\beta} \ar[r] & \LL_{H_\beta}[1].}
$$
We can fill in the dotted arrow, giving (by the long exact sequence in cohomology) a perfect obstruction theory for $P$. We have not checked that this is the same as the one given by Proposition \ref{PTloc}, though it surely is. But the complexes $E\udot_{red}$ and $(E_{X,red}\udot|_{P_S})^{fix}$ have the same K-theory classes, since they sit inside exact triangles with the same objects. This is enough to ensure that the reduced virtual cycles are the same.

\begin{theorem}
The constructions of this Appendix define a reduced virtual cycle
$$
[P_{1-h+n}(S,\beta)]^{red}\in H_{2v}(P_{1-h+n}(S,\beta)),\qquad
v=h-1+n+\int_\beta c_1(S)+h^{0,2}(S),
$$
whenever Condition \eqref{condition} holds. It coincides with the reduced class \eqref{Predvd} when Condition \eqref{condition2} also holds. When Condition \eqref{assume} holds, its pushforward to the smooth ambient space $H_\gamma\times S^{[n]}$ is
\beq{KT2}
c_{top}(F)\,.\,c_n\big(\O(\mathcal D-\pi^*A)^{[n]}\big),
\eeq
where $F$ is the bundle \eqref{vect}.
\end{theorem}

\begin{proof}
Assuming \eqref{assume}, we can summarise the results of this Appendix as follows. The perfect obstruction theory $E\udot_{red}$ arose from cutting $P$ out of $H_\gamma\times S^{[n]}$ firstly by a section $\sigma$ of the vector bundle $F$ \eqref{vect} (pulled back from $H_\gamma$) followed by a section $\sigma_{\mathcal C}$ of the bundle $\O(\mathcal C)^{[n]}$ over $H_\beta\times S^{[n]}$ \eqref{budl}. The latter bundle extends over $H_\gamma\times S^{[n]}$ as the bundle $\O(\mathcal D-\pi_S^*A)^{[n]}$, even though the section $\sigma_{\mathcal C}$ does not. This is enough to give \eqref{KT2}.

The assumption \eqref{assume} that $H^2(L)=0$ for all $L\in\Pic_\beta(S)$ was required to ensure that the sheaf $F$ of \eqref{vect} was a vector bundle on all of $H_\gamma$. If instead we impose only Condition \eqref{condition} -- that $H^2(L)=0$ for all \emph{effective} $L\in\Pic_\beta(S)$ -- then $F$ is a sheaf on $H_\gamma$ which is locally free on $H_\beta$. It is therefore locally free in a neighbourhood of $H_\beta\subset H_\gamma$.

The global description of $H_\beta$ as the zero locus of the section $\sigma$ of $F$ still holds. Since $F$ is a bundle near $H_\beta$, this is enough to induce the perfect obstruction theory on $H_\beta$. The construction of the reduced class then follows as before.

Since the perfect obstruction theory $E\udot_{red}$ has the same K-theory class as $(E_{X,red}\udot|_{P_S})^{fix}$, when Condition \eqref{condition2} also holds we get the same reduced cycle as we do from the perfect obstruction theory of Proposition \ref{PTloc}.
\end{proof}

In \cite{KT2} we use the global description as the zero locus of a section of a bundle to do computations starting from the formula \eqref{KT2}. A priori then, these calculations require the stronger hypothesis \eqref{assume}.

\medskip
\noindent {\tt{mkool@math.ubc.ca}} \\
\noindent {\tt{richard.thomas@imperial.ac.uk}} \\
\noindent {\tt{dmitri.panov@kcl.ac.uk}} \\

\end{document}